\documentclass[preprint,12pt]{elsarticle}

\usepackage{graphicx}
\usepackage{epsfig}
\usepackage{amssymb}
\usepackage{amsmath}
\usepackage{dsfont}
\usepackage{pb-diagram}

\usepackage{amsthm}
\newtheorem{thm}{Theorem}[section]
\newtheorem{lem}[thm]{Lemma}

\newtheorem{prop}[thm]{Proposition}
\newtheorem{cor}[thm]{Corollary}
\newdefinition{defn}[thm]{Definition}
\newdefinition{exmp}[thm]{Example}
\newtheorem{proc}[thm]{Procedure}
\newdefinition{rem}[thm]{Remark}

\begin{document}

\begin{frontmatter}

\title{The quantized walled Brauer algebra\\and mixed tensor space}

\author[s]{R. Dipper }
\ead{rdipper@mathematik.uni-stuttgart.de}
\author[c]{S. Doty }
\ead{doty@math.luc.edu}
\author[s]{F. Stoll }
\ead{stoll@mathematik.uni-stuttgart.de}

\address[s]{ Institut f\"ur Algebra und Zahlentheorie,
  Universit\"at Stuttgart,
  Pfaffenwaldring 57, 70569 Stuttgart,
  Germany}
\address[c]{  Department of Mathematics and Statistics,
  Loyola University Chicago, 
  1032 W.~Sheridan Road, 
  Chicago, IL 60660 USA}

\begin{abstract}
  In this paper we investigate a multi-parameter deformation
  $\mathfrak{B}_{r,s}^n(a,\lambda,\delta)$ of the walled Brauer
  algebra which was previously introduced by Leduc (\cite{leduc}). We
  construct an integral basis of
  $\mathfrak{B}_{r,s}^n(a,\lambda,\delta)$ consisting of oriented
  tangles which is in bijection with walled Brauer diagrams. Moreover,
  we study a natural action of $\mathfrak{B}_{r,s}^n(q)=
  \mathfrak{B}_{r,s}^n(q^{-1}-q,q^n,[n]_q)$ on mixed tensor space and
  prove that the kernel is free over the ground ring $R$ of rank
  independent of $R$. As an application, we prove one side of
  Schur--Weyl duality for mixed tensor space: the image of
  $\mathfrak{B}_{r,s}^n(q)$ in the $R$-endomorphism ring of mixed
  tensor space is, for all choices of $R$ and the parameter $q$, the
  endomorphism algebra of the action of the (specialized via the
  Lusztig integral form) quantized enveloping algebra $\mathbf{U}$ of
  the general linear Lie algebra $\mathfrak{gl}_n$ on mixed tensor
  space. Thus, the $\mathbf{U}$-invariants in the ring of $R$-linear
  endomorphisms of mixed tensor space are generated by the action of
  $\mathfrak{B}_{r,s}^n(q)$.
\end{abstract}


\end{frontmatter}

\input epsf

\section*{Introduction}
Schur--Weyl duality is a special case of a bicentralizer property. For
algebras $A$ and $B$, we say that an $A$-$B$-bimodule $T$ satisfies
the \emph{bicentralizer property} if $\mathrm{End}_A(T)=\rho(B)$ and
$\mathrm{End}_B(T)=\sigma(A)$, where $\rho$ and $\sigma$ are the
corresponding representation maps. The classical Schur--Weyl duality
due to Schur \cite{schur} is the bicentralizer property for the
representations of the group algebras of the symmetric group
$\mathfrak{S}_m$ and the general linear group
$\mathrm{GL}_n(\mathbb{C})$ on \emph{ordinary tensor space}
$V^{\otimes m}$, where $V=\mathbb{C}^n$ is the natural
$\mathbb{C}\mathrm{GL}_n(\mathbb{C})$-module.  The image
$\rho(\mathbb{C}\mathrm{GL}_n(\mathbb{C}))$ in
$\mathrm{End}_{\mathbb{C}}(V^{\otimes m})$ is the ordinary Schur
algebra $S_{\mathbb{C}}(n,m)$ and the finite dimensional
representations of these finite dimensional algebras are precisely the
$m$-homogeneous polynomial representations of
$\mathrm{GL}_n(\mathbb{C})$. These statements continue to hold when
$\mathbb{C}$ is replaced by an arbitrary infinite field \cite{DP,
  jagreen}, and have been further generalized to the case where the
symmetric groups are replaced by corresponding quantum groups and
Hecke algebras associated with symmetric groups (see
\cite{dipperdonkin, duparshallscott}).

This paper, a revision of an earlier preprint that has been in
circulation since 2008, studies Schur--Weyl duality on mixed tensor
space in the following more general setting. Let $R$ be a commutative
ring with $1$, $q$ an invertible element of $R$, $n$ a positive
integer and $r,s$ nonnegative integers.  The (specialized) quantum
group $\mathbf U = \mathbf U_R$ of the general linear group is a
$q$-deformation of the hyperalgebra of the general linear group and
acts naturally on \emph{mixed tensor space} $V^{\otimes r}\otimes
{V^*}^{\otimes s}$, where $V=R^n$ is the natural $\mathbf U$-module
and $V^*$ its dual.  The image $S_{R,q}(n;r,s)$ in
$\mathrm{End}_{R}(V^{\otimes r}\otimes {V^*}^{\otimes s})$ is called
the rational $q$-Schur algebra (studied in the $q=1$ case in
\cite{dipperdoty}) and its finite dimensional representatations are
the same as the rational representations of $\mathbf U$ (or quantum
$\mathrm{GL}_n$) of bidegree $r,s$.  In this paper we show that
$\mathrm{End}_{\mathbf U}(V^{\otimes r}\otimes {V^*}^{\otimes s})$ is
the image of a $q$-deformation $\mathfrak{B}_{r,s}^n(q)$ over $R$ of
the walled Brauer algebra with parameter $n$ acting on mixed tensor
space.  In other words, we show that the algebra of $\mathbf
U$-invariants of $\mathrm{End}_{R}(V^{\otimes r}\otimes {V^*}^{\otimes
  s})$ is generated by the action of $\mathfrak{B}_{r,s}^n(q)$.  The
walled Brauer algebra over $\mathbb C$ was independently 
introduced in  \cite{turaev} and \cite{koike} and its description in
terms of Brauer diagrams was given in \cite{bchlls}. 
A $q$-deformation of the walled Brauer algebra
$\mathfrak{B}_{r,s}^n(q)$
 was investigated for the
case $R=\mathbb{C}(q)$ in \cite{kosudamurakami,leduc}.

Various earlier results are special cases of our results. When the
algebras involved are semisimple, it is possible to obtain the above
description of invariants by decomposing mixed tensor space into
irreducible modules. We obtain more general results by defining an
integral isomorphism between $\mathrm{End}_{\mathbf U}(V^{\otimes
  r}\otimes {V^*}^{\otimes s})$ and $\mathrm{End}_{\mathbf
  U}(V^{\otimes r+s})$.  As a consequence of our main result we obtain
that the kernel of the action of $\mathfrak{B}_{r,s}^n(q)$ on mixed
tensor space is free as $R$-module of rank independent of $R$ and $q$.
Furthermore, by specializing $q$ to 1, we obtain the result that the
algebra of $U_R(\mathfrak{gl}_n)$-invariants of
$\mathrm{End}_{R}(V^{\otimes r}\otimes {V^*}^{\otimes s})$ is
generated by the action of the walled Brauer algebra over $R$. This is
in itself a new result, except in the case $R = \mathbb{C}$, where it
was established in \cite{bchlls}. (For another approach to the $q=1$
result, see the recent preprint \cite{tange}.)

The walled Brauer algebra  has a graphical basis
consisting of certain Brauer diagrams called walled Brauer diagrams.
Leduc \cite{leduc} introduced a multi-parameter deformation
$\mathcal{A}_{r,s}$ of the walled Brauer algebra by (non diagrammatic)
generators and relations and defined an epimorphism from a certain
tangle algebra to $\mathcal{A}_{r,s}$.  This tangle algebra is defined
in a similar way  to Kauffman's tangle algebra \cite{kauffman}.
We alter this description to
make it precise: instead of nonoriented tangles, we take oriented
tangles generating a suitable deformation of the walled Brauer algebra
in terms of diagrams.  These oriented tangles allow us to prove our
results by a simple argument, once it is shown that they act on the
mixed tensor space.

The other half of Schur--Weyl duality for mixed tensor space, namely
that $\mathrm{End}_{\mathfrak{B}_{r,s}^n(q)}(V^{\otimes r}\otimes
{V^*}^{\otimes s})$ is equal to the image of the representation of
$\mathbf U$, i.e.\ the rational $q$-Schur algebra $S_q(n;r,s)$, is
shown by the authors in \cite{dipperdotystoll2} using completely
different methods.

Since the first version of our papers have been posted to the arXiv
the walled Brauer algebra and its action on mixed tensor space has
drawn considerably interest (see
e.~g.~\cite{coxvisscher,coxvissxherdotymartin,donkin,
  goodmangraber,tange,wangkoenig}). In
a 
series of papers, Brundan and Stroppel investigate Khovanov's diagram
algebra
(\cite{brundanstroppel1,brundanstroppel2,brundanstroppel3,
  brundanstroppel4})
and study subsequently an application of these on walled Brauer
algebras in the classical situation of the complex field. In their
fundamental paper \cite{brundanstroppel} they prove in particular
Schur-Weyl duality of a graded version of the walled Brauer algebra
over the complex field in a supergroup setting.
The authors would like to thank the referees for pointing this out and
further  helpful
comments. 

\section{Oriented tangles}\label{section:tangles}
Throughout the paper let $n$ be a positive integer, $r$ and $s$ be
nonnegetive integers and set $m=r+s$. Let $R$ be a commutative ring
with $1$ and $q$ be an invertible element of $R$.

In \cite{leduc} Leduc defined a multi-parameter algebra
$\mathcal{A}_{r,s}$ which is a deformation of the walled Brauer
algebra. In this section we introduce a tangle version
$\mathfrak{B}_{r,s}^n(a,\lambda,\delta)$ of this algebra. We are
primarily interested in the specialized algebra
$\mathfrak{B}_{r,s}^n(q):=\mathfrak{B}_{r,s}^n(q^{-1}-q,q^n,[n]_q)$,
which acts on mixed tensor space by an action which we will describe
later on.

A \emph{tangle} is a knot diagram in a rectangle (with two opposite
edges designated ``top'' and ``bottom'') contained in the plane
$\mathbb{R}^2$, consisting of $m$ vertices along the top edge, $m$
vertices along the bottom edge, and $m$ strands in $\mathbb{R}^3$
connecting the vertices, such that each vertex is an endpoint of
exactly one strand, along with a finite number of closed cycles. Two
tangles are regularly isotopic, if they are related by a sequence of
Reidemeister's moves $\mathrm{RM\,II}$ and $\mathrm{RM\,III}$,
together with isotopies fixing the vertices:
\[
\mathrm {RM\,II}:
\raisebox{-0.5cm}{\epsfbox{reidemeister.1}}
=\raisebox{-0.5cm}{\epsfbox{reidemeister.2}}\quad
\mathrm {RM\,III}:
  \raisebox{-0.5cm}{\epsfbox{reidemeister.3}}=
  \raisebox{-0.5cm}{\epsfbox{reidemeister.4}}\,,\;
  \raisebox{-0.5cm}{\epsfbox{reidemeister.5}}=
  \raisebox{-0.5cm}{\epsfbox{reidemeister.6}}
\]

We fix the following notations: the vertices in the top row are
denoted by $t_1,t_2,\ldots,t_m$, the vertices in the bottom row by
$b_1,b_2,\ldots,b_m$.  An \emph{oriented tangle} is a tangle such that
each of its strands has a direction.  Let $T$ be an oriented
tangle. To $T$ we associate sequences $I=(I_1,\ldots,I_m)$ and
$J=(J_1,\ldots,J_m)$ of symbols '$\downarrow$' and '$\uparrow$' by the
following rules:
\begin{enumerate}
\item
  If there is a strand starting in   $t_k$ then let $I_k=\downarrow$,
  if there is a strand ending in $t_k$, let $I_k=\uparrow$.  
\item 
  If there is a strand starting in   $b_k$ then let $J_k=\uparrow$,
  if there is a strand ending in $b_k$, let $J_k=\downarrow$. 
\end{enumerate}

Observe that for a given oriented tangle the number of entries equal
to $\downarrow$ in $I$ necessarily is the same as the number of
entries in $J$ equal to $\downarrow$ (similarly for $\uparrow$).  We
henceforth assume that the indices $I,J$ satisfy this condition and
call $T$ an oriented tangle of \emph{type} $(I,J)$ (see
Figure~\ref{figure:oriented_tangle}).  We indicate repeated
occurrences of a symbol by exponents.

\begin{figure}[h]
  \centerline{\epsfbox{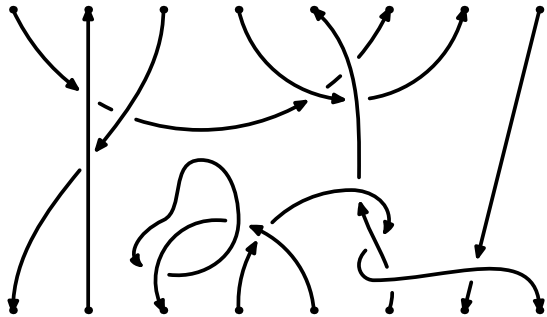}}
  \caption{an oriented tangle of type 
    $((\downarrow,\uparrow,\downarrow^2,\uparrow^3,\downarrow),
    (\downarrow,\uparrow,\downarrow,\uparrow^3,\downarrow^2))$}
  \label{figure:oriented_tangle}
\end{figure}

If $v$ is a vertex, let $s(v)$ be the strand starting or ending in
$v$ with respect to the orientation.
If $s$ is a strand, let $b(s)$ be the beginning (or starting)
vertex and
$e(s)$ the ending vertex of $s$.  Note that strands begin and end at
vertices and are not closed cycles.

\begin{defn}
  \begin{enumerate}
  \item
    Let $\Lambda=\mathbb{Z}[a,\lambda, \lambda^{-1},
    a^{-1}(\lambda^{-1}-\lambda)]$ and let
    $\delta=a^{-1}(\lambda^{-1}-\lambda)$. Let   $\mathcal{U}_{I,J}'$ be the
    $\Lambda$-module generated by  
    the  oriented tangles of type $(I,J)$ up to regular isotopy and
    the following relations 
    which can be applied to a local disk in an oriented tangle 
    \[
    \begin{array}{c@{\qquad}c}
      (O1)& \raisebox{-.4cm}{\epsfbox{relations.1}}- 
      \raisebox{-.4cm}{\epsfbox{relations.2}}=a
      \raisebox{-.4cm}{\epsfbox{relations.3}}\\
      (O2)& \raisebox{-.4cm}{\epsfbox{relations.4}}
      =\raisebox{-.4cm}{\epsfbox{relations.5}}=\delta
      \\
      (O3)&\raisebox{-.4cm}{\epsfbox{relations.6}}=
      \raisebox{-.4cm}{\epsfbox{relations.7}}=\lambda\;
      \raisebox{-.4cm}{\epsfbox{relations.11}}\\
      (O4)&\raisebox{-.4cm}{\epsfbox{relations.8}}=
      \raisebox{-.4cm}{\epsfbox{relations.9}}=\lambda^{-1}\;
      \raisebox{-.4cm}{\epsfbox{relations.11}}\\
    \end{array}
    \]
    If $\Lambda\to R$ is a ring homomorphism, let $\mathcal{U}_{I,J}'$
    be the $R$-module given by the same generators and relations
    with specialized coefficients. 
  \item
    In particular, if $R$ is a commutative ring with $1$ and
    $q\in R$ invertible,  let  $[n]_q=\frac{q^n-q^{-n}}{q-q^{-1}}=
    \sum_{i=0}^{n-1}q^{2i-n+1}$ for the positive integer $n$. Then let 
    $\mathcal{U}_{I,J}$ be  $\mathcal{U}_{I,J}'$ for the
    specialization $\Lambda\to R:a\mapsto q^{-1}-q$,
    $\lambda\mapsto q^n$, $\delta\mapsto[n]_q$. 
  \end{enumerate}
\end{defn}
Note that although $\delta$ is not an independent parameter, the
images of $a$ and $\lambda$ do not uniquely define a ring
homomorphism (e.~g.~if $a\mapsto0$ and $\lambda\mapsto 1$),
thus it is necessary to specify the image of $\delta$. 

The relations in the $R$-module $\mathcal{U}_{I,J}$ are
\[
\begin{array}{c@{\qquad}c}
  (O1)& \raisebox{-.4cm}{\epsfbox{relations.1}}- 
  \raisebox{-.4cm}{\epsfbox{relations.2}}=(q^{-1}-q)
  \raisebox{-.4cm}{\epsfbox{relations.3}}\\
  (O2)& \raisebox{-.4cm}{\epsfbox{relations.4}}
  =\raisebox{-.4cm}{\epsfbox{relations.5}}=
  [n]_q\\
  (O3)&\raisebox{-.4cm}{\epsfbox{relations.6}}=
  \raisebox{-.4cm}{\epsfbox{relations.7}}=q^n\;
  \raisebox{-.4cm}{\epsfbox{relations.11}}\\
  (O4)&\raisebox{-.4cm}{\epsfbox{relations.8}}=
  \raisebox{-.4cm}{\epsfbox{relations.9}}=q^{-n}\;
  \raisebox{-.4cm}{\epsfbox{relations.11}}\\
\end{array}
\]

Let $S$ be an oriented tangle of type $(I,J)$ and $T$ be an oriented
tangle of type 
$(J,K)$. Then all the strands at the bottom vertices of $S$ have the
same direction (up or down) as the corresponding strands at the top
row of $T$. Thus if one places $S$ above $T$ and identifies the bottom
row of vertices in $S$ with the top row in $T$, one gets an oriented tangle of
type $(I,K)$. Denote this concatenated oriented 
 tangle by $S|T$. Concatenation
induces $R$-linear maps $\mathcal{U}_{I,J}'\times
\mathcal{U}_{J,K}'\to \mathcal{U}_{I,K}'$ and  $\mathcal{U}_{I,J}\times
\mathcal{U}_{J,K}\to \mathcal{U}_{I,K}$. In particular,
$\mathcal{U}_{I,I}'$ and $\mathcal{U}_{I,I}$ are
associative  $R$-algebras respectively,
the multiplication given by concatenation
of oriented tangles. 

\begin{defn}\label{defn:qwalledbrauer}
  Let
  $\mathfrak{B}_{r,s}^n(a,\lambda,\delta):=\mathcal{U}_{(\downarrow^r
    \uparrow^s),(\downarrow^r \uparrow^s)}'$ and
  $\mathfrak{B}_{r,s}^n(q):=\mathcal{U}_{(\downarrow^r
    \uparrow^s),(\downarrow^r
    \uparrow^s)}=\mathfrak{B}_{r,s}^n(q^{-1}-q,q^n,[n]_q)$.  Then
  $\mathfrak{B}_{r,s}^n(a,\lambda,\delta)$ and
  $\mathfrak{B}_{r,s}^n(q)$ are associative $R$-algebras which we call
  the \emph{quantized walled Brauer algebras}.
\end{defn}


Note that for $q=1$ 
relation~$(O1)$ means that we don't have to
distinguish over- and under-crossings, relations~$(O3)$ and $(O4)$ are
just the Reidemeister Moves~I. We will show that $\mathcal{U}_{I,J}'$
and $\mathcal{U}_{I,J}$ have
$R$-bases indexed by certain oriented tangles which are the
$q$-analogues of Brauer diagrams in the classical case $q=1$.  If $q$
is not necessarily $1$, then using relation~$(O1)$  
one can switch
crossings in an oriented tangle (modulo a linear combination of
oriented tangles with fewer crossings), in particular one can move a
strand on top of all other strands, a second strand on top of the
remaining strands and so on. The relations $(O1)$, $(O3)$ and $(O4)$
can be used to untangle the strands and closed cycles, and finally
relation~$(O2)$ can be used to eliminate the unknotted cycles.  This
serves as motivation for the following definition of descending
oriented tangles, since each oriented tangle can be written as a linear
combination of descending ones (with respect to a given total ordering
on the starting vertices).

\begin{defn}\label{definition:descending}
  Let $T$ be an oriented tangle of type $(I,J)$.  Chose a total
  ordering $\preceq$ on the starting vertices of $T$.  We say that $T$
  is \emph{descending} with respect to $\preceq$ if the following
  conditions hold:
  \begin{enumerate}
  \item  \label{descending_1}No strand crosses itself.
  \item  \label{descending_2}Two strands cross at most once. 
  \item  \label{descending_3}There are no closed cycles.
  \item  \label{descending_4}
    If two strands $s_1$ and $s_2$ with $b(s_1)\prec b(s_2)$
    cross, then $s_1$ over-crosses $s_2$.
\end{enumerate}
\end{defn}
This means that a strand $s$ lies on top of all strands $t$ with
beginning vertex $b(t)\succ b(s)$. Note that each oriented tangle satisfying
\ref{descending_1}, \ref{descending_3} and \ref{descending_4}
is regularly isotopic to a descending oriented tangle.


We will show that the descending oriented tangles (up to regular isotopy) with
respect to some fixed total ordering on the starting vertices form 
bases for $\mathcal{U}_{I,J}'$ and $\mathcal{U}_{I,J}$.

A \emph{Brauer diagram} is an (unoriented) graph consisting of $2m$ vertices
arranged in two rows and $m$ edges, such that each vertex is an
endpoint of exactly one edge.  If $T$ is an oriented tangle, let
$c(T)$ be the Brauer diagram obtained by connecting the vertices which
are connected by a strand in the oriented tangle $T$, so closed
circles are ignored.  We call $c(T)$ the \emph{connector} of $T$.  We
say, that a Brauer diagram is of type $(I,J)$, if it is the connector
of an oriented tangle of type $(I,J)$.  Note that for a given Brauer
diagram $c$, there are always different indices, such that $c$ is the
connector of oriented tangles of these types.

Furthermore, there are precisely $m!$ Brauer diagrams of type $(I,J)$:
For a given type $(I,J)$ there are $m$ starting vertices and $m$ end
vertices.  In a Brauer diagram of type $(I,J)$ the edges always
connect a starting vertex with an ending vertex. Thus the Brauer
diagrams of type $(I,J)$ are in one-to-one correspondence with
bijections from the set of starting vertices to the set of ending
vertices, and there are $m!$ such bijections.

Fix some total ordering on the set of starting vertices. For each
Brauer diagram $c$ of type $(I,J)$ choose an oriented tangle $T_c$ of type
$(I,J)$ with connector $c$, which is descending with respect to the
fixed total ordering.  Consider the strand in $T_c$ starting at the
first starting vertex. The Brauer diagram $c$ determines the end
vertex of this strand. Furthermore, this strand does not cross itself,
thus if $T_c'$ is another descending oriented tangle with connector $c$, one
can use the Reidemeister Moves II and III to move the top strand of
$T_c'$ to the position of the top strand of $T_c$. A similar
argument works for the other strands, thus $T_c$ and $T_c'$ are
regularly isotopic and we have a one-to-one correspondence between the
set of Brauer diagrams of type $(I,J)$ and the set of isotopy classes
of descending oriented tangles of type $(I,J)$ with respect to a given total
ordering of the starting vertices.

\begin{lem}{(\cite{weimer})}\label{lem:basis} 
  Fix a total ordering on the set of starting vertices and let
  $\mathcal{B}=\{T_c\}$ where $c$ runs
  through the set of Brauer diagrams of type $(I,J)$.
  Then
  $\mathcal{U}_{I,J}'$ is $R$-free of rank $m!$ with basis
    $\mathcal{B}$ for any specialization $\Lambda\to R$. In particular
    $\mathcal{U}_{I,J}$ is $R$-free with basis $\mathcal{B}$.
  \end{lem}
\begin{proof}
  In his diploma thesis~\cite{weimer},  F.~Weimer  showed the special
  case $I=J=(\downarrow^r\uparrow^s)$ for $R=\Lambda$.
  Actually, the restriction to $(\downarrow^r\uparrow^s)$ is not
  essential, if $I$ and $J$ have $r$ entries equal to $\downarrow$ and
  $s$ entries equal to $\uparrow$, then $\mathcal{U}_{I,J}$ and 
  $\mathcal{U}_{(\downarrow^r
    \uparrow^s),(\downarrow^r \uparrow^s)}$ are isomorphic as
  $R$-modules, the isomorphism is given by pre- and post-multiplication
  with appropriate invertible oriented tangles. 

   The proof of the lemma follows exactly  the proof of Morton and Wassermann
   (\cite{mortwas})  
   for the Birman--Wenzl--Murakami algebra.
   Instead of the knot invariant Morton and Wassermann used, this
   proof needs a knot invariant of oriented knots, whose existence is
   assured in  \cite{homfly}. With this invariant, the linear
   independence of $\mathcal{B}$ can be shown. 
   The specialized case follows from change
   of base rings. 
\end{proof}
\begin{cor}\label{cor:basis}
  For each Brauer diagram c of type $(I,J)$ choose an oriented  tangle
  $S_c$ of type $(I,J)$ with
  connector $c$ without self-crossings and closed cycles such that two
  strands cross at most once. Then the $S_c$ form a basis of
  $\mathcal{U}'_{I,J}$ resp.~of $\mathcal{U}_{I,J}$. 
\end{cor}
\begin{proof}
  Using relation $(O1)$, 
  each such oriented tangle can be written as the same oriented tangle with the
  orientation of certain crossings changed plus a linear combination of
  oriented tangles with fewer crossings.
  Inductively, each $S_d$ can be written as $T_d$ plus a 
  linear combination of $T_c$'s such that $c$ has fewer crossings than
  $d$. Choose an ordering on the set of Brauer diagrams of type
  $(I,J)$ that is compatible with the number of crossings. Let $M$ be
  the matrix whose columns consist of the coefficients of these linear
  combinations. Then $M$ is triangular with ones on the diagonal, and
  thus is invertible. It follows that $\{S_c\}$ is as well a basis and
  $M$ is the base change matrix.  
\end{proof}

\begin{defn}\label{defn:Hecke}
  If $R$ is a commutative ring with $1$ and $q\in R$ is invertible then
  let $\mathcal{H}_m$ be the associative $R$-algebra with identity generated by
  $T_1,\ldots,T_{m-1}$ subject to the relations
  \[
  \begin{array}{rcll}
    (T_i+q)(T_i-q^{-1})&=&0&\text{ for }i=1,\ldots,m-1\\
    T_iT_{i+1} T_i&=&T_{i+1} T_iT_{i+1}&\text{ for }i=1,\ldots,m-2\\
    T_iT_j&=&T_jT_i &\text{ for }|i-j|\geq 2.
  \end{array}
  \]
  The algebra $\mathcal{H}_m$ is called the
  \emph{Hecke algebra}.
\end{defn}

We remark that the first defining relation $(T_i+q)(T_i-q^{-1})=0$
above is often replaced by a relation of the form
$(\tilde{T_i}+\tilde{q})(\tilde{T_i}-1)=0$. The two approaches are
easily seen to be equivalent, via the transformation given by
$\tilde{T_i}=qT_i$, $\tilde{q}=q^2$.

\begin{cor}
  $\mathfrak{B}_{m,0}^n(q)$  and $\mathcal{H}_m$
  are isomorphic as $R$-algebras.
\end{cor}
\begin{proof}
  The generators 
  $\raisebox{-.1cm}{\epsfbox{si.30}}$ of $\mathfrak{B}_{m,0}^n(q)$
  satisfy the defining relations
  of the $T_i\in \mathcal{H}_m$.
  Thus $\mathfrak{B}_{m,0}^n(q)$ is an epimorphic
  image of $\mathcal{H}_m$. For $w\in\mathfrak{S}_m$ let
  $w=s_{i_1}s_{i_2}\ldots s_{i_l}$ be a reduced expression in terms of
  Coxeter generators and let
  $T_w= T_{i_1}T_{i_2}\ldots T_{i_l}$. It is known that $T_w$ does not
  depend on the reduced expression and $\{T_w\mid w\in\mathfrak{S}_m
  \}$ is a basis of $\mathcal{H}_m$. The images of these basis
  elements form a basis of
  $\mathfrak{B}_{m,0}^n(q)$ by Corollary~\ref{cor:basis}.
\end{proof}

Now we can connect our construction and results to those of Leduc
and Kosuda--Murakami.

\begin{defn}[\cite{leduc}]
Let $\mathcal{A}_{r,s}$ be the  $\Lambda$-algebra
generated by the elements
$g_1,\ldots,g_{r-1}$, $g_1^*,\ldots,g_{s-1}^*$ and $D$ subject to the
following relations
\[
\begin{array}{cl@{\qquad}cl}
 (\mathrm i)&g_ig_j=g_jg_i, |i-j|\geq 2&
 (\mathrm i^*)&g_i^*g_j^*=g_j^*g_i^*, |i-j|\geq 2\\
 (\mathrm{ii})&g_ig_{i+1}g_i=g_{i+1}g_ig_{i+1}&
 (\mathrm{ii}^*)&g_j^*g_{j+1}^*g_j^*=g_{j+1}^*g_j^*g_{j+1}^*\\
 (\mathrm{iii})&g_i^2+ag_i-1=0&
 (\mathrm{iii}^*)&{g_j^*}^2+ag_j^*-1=0\\
 (\mathrm{iv})&g_ig_j^*=g_j^*g_i\\
 (\mathrm{v})&Dg_i=g_iD, 2\leq i\leq r-1&
 (\mathrm{v}^*)&Dg_j^*=g_j^*D, 2\leq j\leq s-1\\
 (\mathrm{vi})&Dg_1D=\lambda^{-1} D&
 (\mathrm{vi}^*)&Dg_1^*D=\lambda^{-1} D\\
 (\mathrm{vii})&D^2=\delta D\\
 (\mathrm{viii})&Dg_1^{-1}g_1^*Dg_1=Dg_1^{-1}g_1^*Dg_1^* &
 (\mathrm{viii}^*)&g_1Dg_1^{-1}g_1^*D=g_1^*Dg_1^{-1}g_1^*D
\end{array}
\]
Note that  $(\mathrm{iii})$  implies that $g_1$ is invertible with
$g_1^{-1}=g_1+a$.

If $f:\Lambda\to R$ is a specialization homomorphism, let
$\mathcal{A}_{r,s}(R)$ be the algebra given by the same generators and
relations with specialized coefficients. 
\end{defn}
The definition of Leduc in \cite{leduc} used a different set of
parameters. One may recover Leduc's algebra as the specialization
$\Lambda\to\mathbb{C}(z,q): a\mapsto q-q^{-1}, \lambda\mapsto z^{-1},
\delta\mapsto (z-z^{-1})/(q-q^{-1})$.

\begin{cor}
  For any specialization homomorphism $\Lambda\to R$, the $R$-algebras 
  $\mathfrak{B}_{r,s}^n(a,\lambda,\delta)$
  and $\mathcal{A}_{r,s}(R)$
  are isomorphic. 
  In particular
  for $R=\mathbb{C}(q)$,
  $\mathfrak{B}_{r,s}^n(q)$ is isomorphic to the algebra
  $H_{r,s}^n(q)$ in
  \cite{kosudamurakami}.
\end{cor}

\begin{proof}
  It is easy to verify that for any specialization
  there is a well defined  homomorphism of
  algebras  $\mathcal{A}_{r,s}(R)\to 
  \mathfrak{B}_{r,s}^n(a,\lambda,\delta)$ given by
  \begin{alignat*}{2}
    &1\mapsto\raisebox{-.1cm}{\epsfbox{walledtangles.4}}
    \quad&g_i\mapsto \raisebox{-.1cm}{\epsfbox{walledtangles.2}}\\
    &D\mapsto\raisebox{-.1cm}{\epsfbox{walledtangles.1}}
    \qquad&g_j^*\mapsto \raisebox{-.1cm}{\epsfbox{walledtangles.3}}
  \end{alignat*}
  where the crossings affect the $r-i$-th and $r-i+1$-st vertex  and the 
  $r+j$-th and $r+j+1$-st vertex respectively.

  The algebra $\mathcal{A}_{r,s}(R)$ has a basis introduced in
  \cite{leduc} called the standard monomial basis. Let
  $\mathfrak{S}_r$ and
  $\mathfrak{S}_s$ be the symmetric groups
  on $r$ and $s$ letters respectively. Let  $w=s_{x_1}\cdots
  s_{x_l}\in  \mathfrak{S}_r$ and $v= s_{y_1}\cdots s_{y_k}\in
  \mathfrak{S}_s$ be reduced words in terms of Coxeter
  generators.  Furthermore let $0\leq N\leq r,s$ and let $0\leq i_1\leq
  i_2\leq \ldots\leq i_N\leq s-1$ and $0\leq j_N\leq
  j_{N-1}\leq \ldots\leq j_1\leq r-1$. 
  Then to this data is attached a basis element in the following way:
  Let $g_w= g_{x_1}\cdots g_{x_l}$ and $g_v^*= g_{y_1}^*\cdots g_{y_k}^*$.
  Note that in \cite{leduc}, the reduced expressions are of a
  particular form, but choosing another reduced expression will lead
  to the same element. Let $W_i=g_ig_{i-1}\cdots g_1$ and
  $W_j^*=g_j^*g_{j-1}^*\cdots g_1^*$. Then the attached basis element
  is the element
  \[
  g_w  W_{i_1}^*DW_{j_1}^{-1}  W_{i_2}^*DW_{j_2}^{-1} \cdots
  W_{i_N}^*DW_{j_N}^{-1}   g_v^*. 
  \]
  Consider the image of such a basis element under the algebra
  homomorphism  $\mathcal{A}_{r,s}(R)\to 
  \mathfrak{B}_{r,s}^n(a,\lambda,\delta)$.  The connector of this
  image is the Brauer 
  diagram obtained by replacing each over- or undercrossing by a
  crossing. The numbers $i_1+1,\ldots,i_N+1$ indicate the ending vertices
  of horizontal edges starting at the first $r$ vertices in the top
  row in this Brauer diagram.
  The numbers $j_1+1,\ldots,j_N+1$ similarly indicate the  ending vertices
  of horizontal edges starting at the last $s$ vertices in the bottom
  row. The elements $w$ and $v$ permute these edges respectively. Thus
  different data leads to different connectors of the image of the
  corresponding monomial basis element and each Brauer diagram of type
  $(\downarrow^r\uparrow^s)$ can be obtained in this way. 

  By construction, the image of a standard monomial basis element does
  not contain self-crossings, closed cycles or strands that cross more
  than once. By Corollary~\ref{cor:basis} the algebra homomorphism 
  $\mathcal{A}_{r,s}(R)\to 
  \mathfrak{B}_{r,s}^n(a,\lambda,\delta)$ maps a basis to a basis and
  is thus an isomorphism. 
\end{proof}

Note that in  the classical case, i.~e.~for $q=1$,
the walled Brauer algebra is a
subalgebra of the Brauer algebra. There is also a $q$-deformation of
the Brauer algebra, the BMW-algebra (Birman-Murakami-Wenzl algebra, 
see \cite{birmanwenzl}), which
has also a tangle description. Unlike the quantized walled Brauer
algebra, the tangles used for the BMW-algebra are not oriented and
relation $(O1)$ involves an additional term. Thus, the quantized
walled Brauer algebra is not a subalgebra of the BMW-algebra, at least
not in a canonical way.

\section{The action of  oriented tangles}\label{section:tangleaction} 
From now on, let $R$ be a
commutative ring with $1$ and $q\in R$ invertible.  In this section we
will assign an $R$-linear map to each oriented tangle of type $(I,J)$
considered as an element of $\mathcal{U}_{I,J}$. This assignment is
compatible with concatenation of oriented tangles and composition of
maps. In particular, this map will be used to define an action of
$\mathfrak{B}_{r,s}^n(q)$ on mixed tensor space. This action coincides
with the action of $\mathcal{A}_{r,s}(R)$ given in
\cite{leduc}. However, we will give a procedure to calculate the
entries of the matrix of an arbitrary oriented tangle, and not only of
a generator.

Let $V$ be a free $R$-module of rank $n$ with basis
$\{v_1,\ldots,v_n\}$. Let $V^*=\mathrm{Hom}_R(V,R)$ be the dual module
with dual basis $\{v_1^*,\ldots,v_n^*\}$.  To each $m$-tuple $I$ with
entries in $\{\uparrow,\downarrow\}$ we associate an $m$-fold tensor
product $V_I$, such that the $k$-th component of $V_I$ is $V$ if
$I_k=\downarrow$ and $V^*$ if $I_k=\uparrow$. In particular we define
\begin{defn}
  $V_{(\downarrow^m)}=V^{\otimes m}$ is called \emph{(ordinary) tensor
  space} and 
  $V_{(\downarrow^r\uparrow^s)}=V^{\otimes r}\otimes {V^*}^{\otimes s}$
  is called \emph{mixed tensor space}. 
\end{defn}
At this point it might not be clear why we distinguish between $V$ and
$V^*$. As $R$-modules, ordinary and mixed tensor space are isomorphic.
The reason for this will become appearent when the quantum groups come
into play.

Let $I(n,m)$ be the set of $m$-tuples with entries from
$\{1,\ldots,n\}$. The elements of $I(n,m)$ are called \emph{multi
  indices}.  For $\mathbf i=(i_1,\ldots,i_m)\in I(n,m)$ let
$v_{\mathbf i}^I=x_{i_1}\otimes \ldots \otimes x_{i_m}\in V_I$ where
$x_{i_\rho}$ is either $v_{i_\rho}\in V$ or $v_{i_\rho}^*\in V^*$. Then
$\{v_{\mathbf i}^I\mid \mathbf i\in I(n,m)\}$ is a basis of $V_I$. If
$I=(\downarrow^m)$, i.~e.~$V_I$ is the ordinary tensor space we simply
write $v_{\mathbf i}$ for these basis elements.

We will now assign to each oriented tangle $T\in\mathcal{U}_{I,J}$ a
linear map $\psi_T$ from $V_I$ to $V_J$ acting on $V_I$ from the right
by giving a procedure to calculate the entries of the corresponding
matrix with rows and columns indexed by $I(n,m)$. We let these
matrices act from the right as well.

Let $\mathbf i,\mathbf j\in I(n,m)$.  Write the entries of $\mathbf i$
and $\mathbf j$ in order from left to right along the vertices of the
top and bottom row respectively of the oriented tangle.  If $\beta$ is
a vertex, let $\beta(\mathbf i,\mathbf j)$ be the entry written at the
vertex $\beta$.  We say that $T$ is \emph{descending} with respect to
$(\mathbf i,\mathbf j)$ if $T$ is descending with respect to some
total ordering $\prec$ on the starting vertices, such that
$\beta\prec\gamma$ implies $\beta(\mathbf i,\mathbf j) \leq
\gamma(\mathbf i,\mathbf j)$, that is $\prec$ is a refinement of the
ordering on starting vertices by their attached indices. So this
means, that the strands starting from vertices $\beta$ with
$\beta(\mathbf i,\mathbf j)=1$ are on top of the strands with starting
index $2$, etc.  For example for $n=3$, $m=4$, let
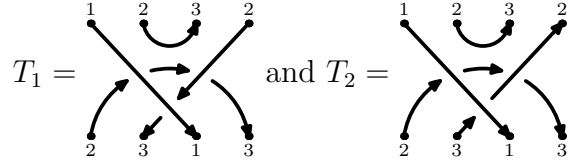
\[
T_1=\raisebox{-1cm}{\epsfbox{desc.1}}\text{ and }
T_2=\raisebox{-1cm}{\epsfbox{desc.2}}
\]
Then $T_1$ 
is descending with respect to
$((1,2,3,2),(2,3,1,3))$. 
The total ordering on the starting vertices can be chosen to be
$t_1\prec
t_4\prec b_1\prec t_2$, $t_1\prec
t_4\prec t_2\prec b_1$ or
$t_1\prec t_2\prec  t_4\prec b_1$.
$T_2$ is not 
descending with respect to $((1,2,3,2),(2,3,1,3))$.

\smallskip
\begin{proc}\label{procedure}
  Let $(I,J)$ be a type and $\mathbf i,\mathbf j\in I(n,m)$ be fixed.
  Choose a total ordering on the set of starting vertices compatible
  with the partial order on the starting vertices by the double index
  $(\mathbf i,\mathbf j)$.

  Let $T$ be an oriented tangle of type $(I,J)$ which is descending
  with respect to this total ordering. To $T$ we define a value
  $\psi_{\mathbf i,\mathbf j}(T)\in R$ as follows.  If there is a
  strand in $T$ whose starting and ending vertex are labeled by different
  indices from the double index $(\mathbf i,\mathbf j)$, let
  $\psi_{\mathbf i,\mathbf j}(T)=0$.  If for all strands in $T$ the
  starting and ending vertex carry the same label, we may assign this
  label to the strand as well. Then let $\psi_{\mathbf i,\mathbf
    j}(T)$ be the product of the following factors (the empty product
  being $1$ as usual):
 \begin{itemize}
 \item
   a factor $q^{-1}$ for each crossing
   $\raisebox{-.1cm}{\epsfbox{crossings.1}}$ of strands labeled by the
   same number and a factor $q$ for each such crossing
   $\raisebox{-.1cm}{\epsfbox{crossings.2}}$,
 \item 
   for each horizontal strand in the upper row from left to right
   labeled by $i$ a factor $q^{2i-n-1}$, and a factor $q^{-2i+n+1}$
   for each horizontal strand in the bottom row from left to right
   labeled by $i$.
 \end{itemize}
 Note that if $T$ and $T'$ are regularly isotopic then 
 $\psi_{\mathbf i,\mathbf j}(T)=\psi_{\mathbf i,\mathbf j}(T')$, and thus
 $\psi_{\mathbf i,\mathbf j}$ is a well defined map from a basis of
 $\mathcal{U}_{I,J}$ to $R$ which extends uniquely to an $R$-linear map
 $\mathcal{U}_{I,J}\to R$.
\end{proc}
This procedure involves a choice, namely the chosen total ordering
on the starting vertices  compatible with the partial ordering
given by $(\mathbf i,\mathbf j)$. Before we show that the resulting
maps $ \psi_{\mathbf i,\mathbf j}$ do not depend on this choice
we illustrate the procedure by an example.
\begin{exmp}
  Let 
  $\mathbf i=(i_1,i_2,i_3),\mathbf j=(j_1,j_2,j_3)$ be multi indices
  with $i_1=2,j_1=1,j_3=1$ and
  $T=\raisebox{-.4cm}{\epsfbox{example.1}}$.
  Then we have
  \[
  T=\raisebox{-.4cm}{\epsfbox{example.1}}
  =\raisebox{-.4cm}{\epsfbox{example.2}}
    +(q-q^{-1})\raisebox{-.4cm}{\epsfbox{example.3}},
  \]
  written as a linear combination of 
  descending oriented tangles. We see that 
  $a_{\mathbf i,\mathbf j}\neq 0$ only if 
  $(\mathbf i,\mathbf j)=((2,1,1),(1,2,1))$ or 
  $(\mathbf i,\mathbf j)=((2,1,2),(1,1,1))$. 
  If $(\mathbf i,\mathbf j)=((2,1,1),(1,2,1))$,
  then $\psi_{\mathbf i,\mathbf j}(T)=q^{-1}$ (for the crossing of the
  strands indexed by $1$). If $(\mathbf i,\mathbf
  j)=((2,1,2),(1,1,1))$, then $\psi_{\mathbf i,\mathbf j}(T)
  =(q-q^{-1})q^{n-1}\cdot
  q^{3-n}=q^3-q$ (one
  horizontal strand from left to right on the bottom  indexed by $1$,
  one such strand on the top, indexed by $2$). 
\end{exmp}

\begin{lem}\label{lem:well-defined}
  For each $\mathbf i,\mathbf j\in I(n,m)$, 
  the map $\psi_{\mathbf i,\mathbf
    j}:\mathcal{U}_{I,J}\to R$ defined in Procedure~\ref{procedure}
  is independent 
  of the choice of 
  the total ordering
  on the starting vertices compatible with the partial ordering
  given by $(\mathbf i,\mathbf j)$.
\end{lem}
\begin{proof}
  Suppose  that we have two different total orderings on the set of 
  starting vertices. Then the first ordering is obtained from the
  second one  by permuting starting vertices with the same index. 
  We may assume that they only differ in the total ordering of two
  starting vertices  
  $\beta_1$ and $\beta_2$ which are adjacent in both of these
  orderings,
  which means that they are indexed by the same number and lie in two
  adjacent layers. Thus if $T$ is a descending tangle with respect to
  one of these orderings, then there is no strand that
  over-crosses  one of the strands starting in $\beta_1$ and $\beta_2$
  and under-crosses the other strand.

  Let $T$ be a descending oriented tangle
  with respect to the first total ordering.
  If $s(\beta_1)$ and $s(\beta_2)$ do not cross,
  then $T$ is descending with respect to the second total ordering and
  $\psi_{\mathbf i,\mathbf  j}(T)$ is the same for both orderings.

  Suppose now that they do cross and the crossing is
  of the kind  $\raisebox{-.1cm}{\epsfbox{crossings.1}}$.
  Let $T_1$ be the tangle obtained by replacing the crossing in $T$ by
  the crossing  $\raisebox{-.1cm}{\epsfbox{crossings.2}}$ and let
  $T_2$ be the tangle obtained by replacing the crossing by
  $\raisebox{-.1cm}{\epsfbox{crossings.12}}$. Then by relation~$(O1)$
  we have  $T=T_1+(q^{-1}-q)T_2$ and $T_1$ and $T_2$ are descending
  with respect to the second ordering. 

  Let $\psi_{\mathbf i,\mathbf  j}$ be calculated with respect to
  the first ordering and $\psi_{\mathbf i,\mathbf  j}'$ with
  respect to the second ordering. 
  Since $T$ and $T_1$ have the same horizotal edges and all crossings
  except for the one crossing coincide, we get 
  $\psi_{\mathbf i,\mathbf  j}'(T_1)=q^2\psi_{\mathbf i,\mathbf
    j}(T)$. 
  Now,  the number
  of horizontal 
  strands on the top from left to right starting at a vertex with index $i$
  minus the number 
  of such strands on the bottom is the same  in $T$ and in $T_2$. Thus 
 $\psi_{\mathbf i,\mathbf  j}'(T_2)=q\psi_{\mathbf i,\mathbf
    j}(T)$. We obtain 
  \[ \psi_{\mathbf i,\mathbf  j}'(T)=\psi_{\mathbf
    i,\mathbf  j}'(T_1+(q^{-1}-q)T_2)=(q^2+(q^{-1}-q)q)
  \psi_{\mathbf i,\mathbf j}(T)= \psi_{\mathbf i,\mathbf j}(T)
  \]
  A similar
  argument works for a crossing
  $\raisebox{-.1cm}{\epsfbox{crossings.2}}$.
\end{proof}

\begin{defn}\label{psiT}
To each element $T\in \mathcal{U}_{I,J}$ we associate an element
$\psi_T\in\mathrm{Hom}_R(V_I,V_J)$, acting from the right on $V_I$, by
\[
(v_{\mathbf i}^I)\psi_T=\sum_{\mathbf j\in I(n,m)}\psi_{\mathbf
  i,\mathbf j}(T)v_{\mathbf j}^J
\] 
\end{defn}
Clearly, $\psi:T\mapsto \psi_T$ is an $R$-linear map
$\mathcal{U}_{I,J}\to \mathrm{Hom}_R(V_I,V_J)$. The $\psi_{\mathbf
  i,\mathbf j}(T)$ are the matrix entries of the matrix corresponding
to $\psi_T$ acting from the right.  The next theorem shows that the
homomorphism $\psi$ is compatible with concatenation of oriented
tangles.

\begin{thm}\label{thm:linking}
  Let $T$ be an oriented tangle of type $(I,K)$ 
  and $S$ be an oriented tangle of type
  $(K,J)$. Then
  \[\psi_{T|S}=\psi_S\circ \psi_T.\]
  Extending linearly, this formula also holds for
  $T\in\mathcal{U}_{I,K}$ and $ S\in\mathcal{U}_{K,J}$. 
\end{thm}
The proof of this theorem is done in several steps. Note that the
theorem is equivalent to the formula $\psi_{\mathbf i,\mathbf j}(T|S)=
\sum_{\mathbf k\in I(n,m)} \psi_{\mathbf i,\mathbf k}(T)\psi_{\mathbf
  k,\mathbf j}(S)$ for all $\mathbf i,\mathbf j\in I(n,m)$.

\begin{defn}
  Let  the following oriented tangles be given where the orientation
  of the vertical strands is omitted and has to be chosen
  suitably. The $\rho$ indicates that the horizontal edges and
  crossings affect the $\rho$-th and $\rho+1$-st node.  
  \[
  \begin{array}{lll}
    E_\rho^{\rightleftarrows}=\raisebox{-.1cm}{\epsfbox{si.1}},
    &\quad E_\rho^{\rightrightarrows}=\raisebox{-.1cm}{\epsfbox{si.2}},
    &\quad E_\rho^{\leftleftarrows}=\raisebox{-.1cm}{\epsfbox{si.3}},\\
    E_\rho^{\leftrightarrows}=\raisebox{-.1cm}{\epsfbox{si.4}},
    &\quad
    S_\rho^{\swarrow\!\!\!\!\!\!\searrow}
    =\raisebox{-.1cm}{\epsfbox{si.5}}, 
    &\quad
    S_\rho^{\nearrow\!\!\!\!\!\!\nwarrow}
    =\raisebox{-.1cm}{\epsfbox{si.6}},\\ 
    S_\rho^{\nearrow\!\!\!\!\!\!\searrow}
    =\raisebox{-.1cm}{\epsfbox{si.7}},
    &\quad
    S_\rho^{\swarrow\!\!\!\!\!\!\nwarrow}
    =\raisebox{-.1cm}{\epsfbox{si.8}}, 
    &\quad\mathds{1} =\raisebox{-.1cm}{\epsfbox{si.99}}
  \end{array}
  \]
  We call these oriented tangles 
  \emph{basic oriented tangles} and call the oriented tangles
  $E_\rho^{\rightrightarrows},E_\rho^{\rightleftarrows},
  E_\rho^{\leftrightarrows},E_\rho^{\leftleftarrows}$ of type $e$, and
  the oriented tangles $S_\rho^{\swarrow\!\!\!\!\!\!\searrow},
  S_\rho^{\nearrow\!\!\!\!\!\!\searrow},
  S_\rho^{\swarrow\!\!\!\!\!\!\nwarrow},
  S_\rho^{\nearrow\!\!\!\!\!\!\nwarrow}$  of type $s$.  
\end{defn}
We clearly have
\begin{lem}\label{lem:basictangles}
  $\mathcal{U}_{I,J}$ is generated by the basic oriented tangles, in the sense
  that each element of $\mathcal{U}_{I,J}$ is a linear combination of
  concatenations of basic oriented tangles. 
\end{lem}
Lemma~\ref{lem:basictangles} shows that it suffices to show 
Theorem~\ref{thm:linking} for $S$ a basic oriented tangle. 

\begin{lem}\label{lem:mupltiplicative_m=2}
If $m=2$ and $T$ and $S$ are basic oriented tangles such that $T|S$ is defined,
then we have 
 \[\psi_{T|S}=\psi_S\circ \psi_T.\]
\end{lem}
\begin{proof}
  Consider for example
  
   \[
   T= \raisebox{-.1cm}{\epsfbox{dotcrossings.3}} 
   =  \raisebox{-.1cm}{\epsfbox{dotcrossings.4}}
   +(q-q^{-1}) \raisebox{-.1cm}{\epsfbox{dotcrossings.7}}\;,\quad
   S=\raisebox{-.1cm}{\epsfbox{dotcrossings.8}}
   \]
   Then $T|S=q^n U$ with $U=\raisebox{-.1cm}{\epsfbox{dotcrossings.9}}$. 
   We have
   \[
   \psi_{\mathbf i,\mathbf k}(T)=
   \left\{
     \begin{array}{cl}
       q&\mbox{ if }i_1=i_2= k_1=k_2\\
       1&\mbox{ if }i_1=k_2\neq i_2=k_1\\
       (q-q^{-1})q^{2i_1-2k_1}&\mbox{ if }i_1=i_2>k_1=k_2\\
       0&\mbox{otherwise}
     \end{array}\right.
   \]
   Furthermore,  $\psi_{\mathbf k,\mathbf
     j}(S)=\delta_{k_1,k_2}\delta_{j_1,j_2}$. Thus
   \begin{eqnarray*}
     \sum_{\mathbf k}\psi_{\mathbf i,\mathbf  k}
     (T)     \psi_{\mathbf k,\mathbf  j}(S)
     &=&\sum_{k_1}\psi_{\mathbf i,(k_1,k_1)}
     (T)\delta_{j_1,j_2}\\
     &=& \left(q+(q-q^{-1})\sum_{k_1<i_1}q^{2(i_1-k_1)} \right)
     \delta_{i_1,i_2}\delta_{j_1,j_2}\\
     &=&q^{2i_1-1}
     \delta_{i_1,i_2}\delta_{j_1,j_2}
     =q^n\cdot \psi_{\mathbf i,\mathbf j}(U). 
   \end{eqnarray*}
   The other cases are proved in a similar way. 
\end{proof}

\begin{lem}\label{lem:multiplicative_basic}
  If $T$ and $S$ are basic oriented tangles for the same $\rho$, then
  $\psi_{T|S}=\psi_S\circ\psi_T$. 
\end{lem}
\begin{proof}
  Let $T'$ and $S'$ be the basic oriented tangles for $m=2$ such that $T$ and
  $S$ are obtained from $T'$ and $S'$ by adding the same vertical
  strands to the left and right of the oriented tangles.
  If $\mathbf i,\mathbf j\in I(n,m)$, let $\epsilon_{\mathbf i,\mathbf
  j}= \delta_{i_1,j_1}\ldots     \delta_{i_{\rho-1},j_{\rho-1}}
    \delta_{i_{\rho+2},j_{\rho+2}}\ldots
    \delta_{i_m,j_m}$ be the product of the given Kronecker deltas. 
 Then 
  \[
  \psi_{\mathbf i,\mathbf k}(T)=\epsilon_{\mathbf i,\mathbf k}
  \psi_{(i_\rho,i_{\rho+1}),(k_\rho,k_{\rho+1})}(T')
  \mbox{ and }
  \psi_{\mathbf k,\mathbf j}(S)=\epsilon_{\mathbf k,\mathbf j}
  \psi_{(k_\rho,k_{\rho+1}),(j_\rho,j_{\rho+1})}(S')
  \]
  and Lemma~\ref{lem:mupltiplicative_m=2} shows that
  \[
    \sum_{\mathbf k}\psi_{\mathbf i,\mathbf k}(T)
    \psi_{\mathbf k,\mathbf j}(S)=
    \epsilon_{\mathbf i,\mathbf j}
    \psi_{(i_\rho,i_{\rho+1}),(j_\rho,j_{\rho+1})}(T'|S')
    =\psi_{\mathbf i,\mathbf j}(T|S). 
  \]
 \end{proof}
Note that for $S=\mathds{1}$, Theorem~\ref{thm:linking} is obviously
true.  Before proving the rest of the theorem, we will show several
lemmas dealing with subcases of the theorem.  The setting will always
be the following:
\begin{itemize}
\item
  $S$ is a basic oriented tangle of type $e$ or $s$, such that
  the $\rho$-th and the $\rho+1$-st vertices are affected by the
  horizontal strands or the crossing. 
  $T$ is an oriented tangle such that $T|S$ is
  defined and such that properties
  \ref{descending_1}.-\ref{descending_3}.~in
  Definition~\ref{definition:descending} hold. 
\item
  $S'$ is the oriented tangle for $m=2$ defined as in the proof of
  Lemma~\ref{lem:multiplicative_basic}. 
\item  $\epsilon_{\mathbf i,\mathbf
    j}= \delta_{i_1,j_1}\ldots     \delta_{i_{\rho-1},j_{\rho-1}}
  \delta_{i_{\rho+2},j_{\rho+2}}\ldots
  \delta_{i_m,j_m}$ is defined as in  Lemma~\ref{lem:multiplicative_basic}.

\item If $\mathbf i$ is a multi index and
  $k,l\in\{1,\ldots,n\}$, we write $\mathbf i_{k,l}$ for the multi index
  obtained from $\mathbf i$ by replacing the $\rho$-th entry by  $k$ and
  the $\rho+1$-st entry by $l$. 
\end{itemize}

\begin{lem}\label{lem:multiplicative_strand}
  Suppose 
  there is a strand in $T$ connecting the vertices $b_\rho$ and
  $b_{\rho+1}$. Then we have $\psi_{T|S}=\psi_S\circ\psi_{T}$. 
\end{lem}
\begin{proof}
  Let $\mathbf i,\mathbf j\in I(n,m)$. 
  Without loss of generality we may assume that $T$ is descending with
  respect to $(\mathbf i,\mathbf j)$. Note that then, $T$ is
  descending with respect to $(\mathbf i,\mathbf j_{k,l})$ for any
  $k,l\in\{1,\ldots,n\}$.
  Let $T'$ be
  the oriented tangle of type $e$ with two vertices in the
  bottom and top row, the same horizontal strand as $T$ and an additional
  top horizontal strand, see Figure~\ref{fig:strand}.  The
  calculations to write $T|S$ as a linear combination of 
  descending oriented tangles  are the same as in
  $T'|S'$. 
  Now
  the lemma follows from the result for $m=2$ in
  Lemma~\ref{lem:mupltiplicative_m=2}. 

  \begin{figure}[h]
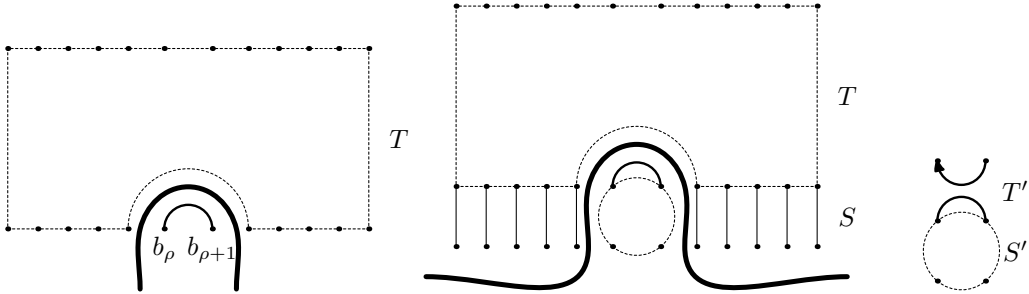

    \centerline{\epsfbox{concat.3}
      \epsfbox{concat.1}\qquad\epsfbox{concat.2}}
    \caption{$T$ has a strand connecting 
      $b_\rho$ and $b_{\rho+1}$}\label{fig:strand}
  \end{figure}
\end{proof}

\begin{lem}\label{lem:multiplicative_noncrossing}
  Fix $\mathbf i,\mathbf j\in I(n,m)$.  Suppose $T$ is an oriented
  tangle such that $b_{\rho}$ and $b_{\rho+1}$ are not connected by a
  strand and the strands ending or starting in $b_{\rho}$ and
  $b_{\rho+1}$ do not cross. Then there is an oriented tangle
  $\tilde{T}$ such that $T-\tilde{T}$ is a linear combination of
  oriented tangles with fewer crossings than $T$ and $\psi_{\mathbf
    i,\mathbf j}(\tilde{T}|S)= \sum_{\mathbf k\in I(n,m)}
  \psi_{\mathbf i,\mathbf k}(\tilde{T})\psi_{\mathbf k,\mathbf j}(S)$.
\end{lem}
\begin{proof}
  If $T$ and $\tilde{T}$ only differ in the orientation of certain
  crossings, then $T-\tilde{T}$ is a linear combination of oriented
  tangles with fewer crossings than $T$. Hence it suffices to show if
  $T$ is descending with respect to some total ordering of the
  starting vertices then the formula in the lemma holds for $T$.

  Assume first, that $S$ is of type $e$. By changing the orientation
  of crossings, we may assume
  the  following: 
  Let $s_1$ be the strand in $T$ ending in $b_{\rho}$ or
  $b_{\rho+1}$,  $s_2$  the strand in $T$ starting at the
  other vertex and $k'$   the label of the starting vertex of $s_1$ 
  when the vertices of $T$ are labeled  by $(\mathbf i,\mathbf j)$. We
  assume that
  $T$ is descending with respect to 
  $(\mathbf i,\mathbf j_{k',k'})$ such that the
  starting vertices of $s_1$ and $s_2$ are adjacent in this ordering. 

  If $j_{\rho}\neq j_{\rho+1}$ then both $\psi_{\mathbf k,\mathbf
    j}(S)$  and $\psi_{\mathbf i,\mathbf j}(T|S)$ vanish since the
  horizontal strands connecting $b_{\rho}$ and $b_{\rho+1}$ in $S$ and
  in $T|S$ stay fixed when the oriented tangles are
  written as a linear combination of
  descending oriented tangles.  So we may 
  assume
  that $j_{\rho}=j_{\rho+1}=j$. 
  If $k_\rho\neq k_{\rho+1}$ then $ \psi_{\mathbf k,\mathbf j}(S)=0$.
  Suppose $k\neq k'$. To write $T$ as a linear combination of
  descending oriented tangles with respect
  to $(\mathbf i,\mathbf j_{k,k})$, one
  can apply the relations to strands 
  which are all on top of $s_1$ if $k<k'$ or below $s_1$ if
  $k>k'$. Therefore,  all oriented tangles appearing in this linear
  combination have 
  the strand $s_1$ whose vertices are indexed by $k$ and $k'$, thus  
  $\psi_{\mathbf i,\mathbf j_{k,k}}(T)= 0$ and we have
  \begin{eqnarray*}
    \sum_{\mathbf k}
    \psi_{\mathbf i,\mathbf k}(T) \psi_{\mathbf k,\mathbf j}(S)
    &=& \sum_{\mathbf k}
    \psi_{\mathbf i,\mathbf k}(T) \epsilon_{\mathbf k,\mathbf j}
    \psi_{(k_\rho,k_{\rho+1}),(j,j)}(S')\\
    &=&\sum_k \psi_{\mathbf i,\mathbf j_{k,k}}(T) 
    \psi_{(k,k),(j,j)}(S')
    =\psi_{\mathbf i,\mathbf j_{k',k'}}(T) 
    \psi_{(k',k'),(j,j)}(S')
  \end{eqnarray*}
  Note that the choice of total ordering on the starting vertices
  assures that $T|S$ is regularly isotopic to an oriented tangle which
  is descending with respect to $(\mathbf i,\mathbf
  j)$.  One can compute the value $\psi_{\mathbf i,\mathbf j}(T|S) $
  directly from $T|S$, since the number of positive crossings minus
  the number of negative crossings is fixed under the Reidemeister
  Moves. 
  Thus the computation of $\psi_{\mathbf i,\mathbf j_{k',k'}}(T) $ differs
  from that of $\psi_{\mathbf i,\mathbf j}(T|S) $ only in the values
  for the strands $s_1$ and $s_2$ in $T$ and the merged strand and the
  bottom horizontal strand  of the $S$-part in $T|S$. Clearly,  
  $\psi_{\mathbf i,\mathbf j_{k',k'}}(T) 
  \psi_{(k',k'),(j,j)}(S')\neq 0$ if and only  if $\psi_{\mathbf
    i,\mathbf j}(T|S) \neq 0$.

 The strands in $T|S$ and in $S'$
  connecting the vertices $b_{\rho}$ and
  $b_{\rho+1}$ yield the same value. Suppose now that the upper
  horizontal strand in $S$ goes from left to right and look at the
  strands $s_1$ and $s_2$ in $T$ and the merged strand in $T|S$, see
  Figure~\ref{fig:horizontalstrands}. 
  \begin{figure}[h]
    \centerline{
      \epsfbox{horizontal.1}\hfil
      \epsfbox{horizontal.2}\hfil
      \epsfbox{horizontal.3}\hfil
      \raisebox{.75mm}{\epsfbox{horizontal.4}}\hfil
      \epsfbox{horizontal.5}\hfil
      \epsfbox{horizontal.6}
    }\caption{}\label{fig:horizontalstrands}
  \end{figure}
  In either case merging the strands $s_1$ and $s_2$ eliminates one
  horizontal strand from left to right on the bottom row or produces an
  additional horizontal  strand from left to right on the upper
  row. Thus we have 
  \[\psi_{\mathbf i,\mathbf j_{k',k'}}(T) 
  \psi_{(k',k'),(j,j)}(S')=\psi_{\mathbf i,\mathbf j}(T|S),\]
  which shows the claim. If the upper  horizontal strand in $S$ goes
  from right  to left, then the number of vertical stands from left to
  right is fixed when merging the strands, and the lemma is also true.

  Now, suppose that $S$ is a basic tangle of type $s$. 
  Let $\mathbf j'$ be the multi index obtained from $\mathbf j$ by
  interchanging the $\rho$-th and the $\rho+1$-st entries. We may
  assume that $T$ is descending with respect to 
  $(\mathbf i,\mathbf j')$. Let $\hat{S}$ be the tangle obtained from
  $S$ by changing the orientation of  the crossing.

  If $S=S_\rho^{\swarrow\!\!\!\!\!\!\searrow}$, let the vertices of
  $T$ be labeled  by $(\mathbf 
  i,\mathbf j)$.
  Let $i$ be the label of the starting vertex 
  of the strand of $T$
  ending in $b_\rho$, and $j$ be the label of the strand
  ending in $b_{\rho+1}$.
  If $i<j$ then $T|S$ is descending with respect to 
  $(\mathbf i,\mathbf j)$ and we have $\psi_{\mathbf i,\mathbf j}(T|S)= 
  \psi_{\mathbf i,\mathbf j'} (T)$. 
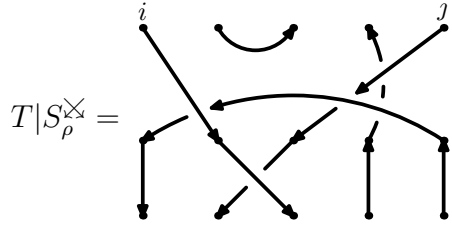
  \[
  T|S_\rho^{\swarrow\!\!\!\!\!\!\searrow}=\;
  \raisebox{-1.25cm}{\epsfbox{composition.3}} 
  \]    
  Furthermore, 
  $\psi_{\mathbf k,\mathbf j}(S)\neq 0$ implies 
  $\mathbf k=\mathbf j'$ or $\mathbf k=\mathbf j$. 
  $\psi_{\mathbf j,\mathbf j}(S)\neq 0$ implies $j_\rho>j_{\rho+1}$,
  thus we can not have $i=j_\rho$ and $j=j_{\rho+1}$, so
  $\psi_{\mathbf i,\mathbf j}(T)\psi_{\mathbf j,\mathbf j}(S)=0$. We have
  \[
  \sum_{\mathbf k}\psi_{\mathbf i,\mathbf k}(T)
  \psi_{\mathbf k,\mathbf j}(S)=
  \psi_{\mathbf i,\mathbf j'}(T)
  \psi_{\mathbf j',\mathbf j}(S)
  =\psi_{\mathbf i,\mathbf j'}(T)
  =\psi_{\mathbf i,\mathbf j}(T|S).
  \]
  If $i>j$, then $T|S=T|\hat{S}+(q^{-1}-q)T$ is a linear combination
  of descending tangles with respect to $(\mathbf i,\mathbf j)$
  and it suffices to show $ \sum_{\mathbf k}\psi_{\mathbf i,\mathbf k}(T)
  \psi_{\mathbf k,\mathbf j}(\hat{S})=\psi_{\mathbf i,\mathbf
    j}(T|\hat{S})$. The proof can be copied from that for $i<j$ by
  reflecting all tangles at a vertical line.

  If $i=j$ then $T|S$ or $T|\hat{S}$ is descending with respect to
  $(\mathbf i,\mathbf j)$,  we assume that $T|S$ is descending. Then  we have 
  $\psi_{\mathbf i,\mathbf j} (T|S)
  =q^{-1}\psi_{\mathbf i,\mathbf j}(T)$. Furthermore,
  $\psi_{\mathbf k,\mathbf j} (S)\neq 0$ only if $\mathbf k=\mathbf j$
  or $\mathbf k=\mathbf j'$ and
  $\psi_{\mathbf
    i,\mathbf j_{k,l}} (T)\neq 0$ only if $k=l$.
  If $j_{\rho}\neq j_{\rho+1}$ we have
  $\sum_{\mathbf k}\psi_{\mathbf i,\mathbf k}(T)
  \psi_{\mathbf k,\mathbf j}(S)=0=\psi_{\mathbf i,\mathbf j}(T|S)$, if
  $j_{\rho}=j_{\rho+1}$, then $\mathbf j=\mathbf j'$ and we have
  \[
  \sum_{\mathbf k}\psi_{\mathbf i,\mathbf k}(T)
  \psi_{\mathbf k,\mathbf j}(S)=
  \psi_{\mathbf i,\mathbf j}(T)
  \psi_{\mathbf j,\mathbf j}(S)=q^{-1} \psi_{\mathbf i,\mathbf j}(T)
  =\psi_{\mathbf i,\mathbf j} (T|S).
  \]
All other cases can be shown using similar arguments. 
\end{proof}
We are now able to prove Theorem~\ref{thm:linking}:

\begin{proof}[Proof of Theorem~{\rm\ref{thm:linking}}]

We will proceed by induction on the number of crossings in
$T$. Suppose there are no crossings in $T$ then $b_{\rho}$ and
$b_{\rho+1}$ are either connected by a strand, or the two strands
starting or ending in $b_{\rho}$ and
$b_{\rho+1}$ do not cross. Thus we can apply
Lemma~\ref{lem:multiplicative_strand} or 
Lemma~\ref{lem:multiplicative_noncrossing}.

Suppose now, that $T$ contains at least one crossing.
Clearly, we can assume that $T$ is descending with respect to some
ordering, thus  two strands cross at most once. If the strands
starting or ending in $b_{\rho}$ and
$b_{\rho+1}$ coincide or do not cross, we can apply
Lemma~\ref{lem:multiplicative_strand} or
Lemma~\ref{lem:multiplicative_noncrossing} to write $T=\tilde{T}+R$
where $\tilde{T}$ satisfies the theorem and $R$ is a linear
combination of oriented tangles with fewer crossings than $T$, and thus
satisfies the theorem  as well by the induction hypothesis.

 So let $T$ be an oriented tangle,
such that these two strands cross. Since two strands never cross more
than once, the Reidemeister Moves III and isotopies fixing the
vertices can be used to move the crossing to the bottom, that is $T$ can
be written as a concatenation $T=\hat{T}|\hat{S}$ where $\hat{S}$ is a basic
oriented tangle with a crossing (or a  crossing with changed orientation)
affecting the
$\rho$-th and $\rho+1$-st vertices and $\hat{T}$
is an oriented tangle with strictly
fewer crossings than $T$. Thus we may apply the induction hypothesis to
$\hat{T}$. Note that $\hat{S}|S$ is 
a linear combination of basic oriented tangles. 
By Lemma~\ref{lem:multiplicative_basic} we get
\[
\psi_{T|S}=
\psi_{\hat{T}|\hat{S}|S}
= \psi_{\hat{S}|S}\circ \psi_{\hat{T}}
=\psi_{S}\circ\psi_{\hat{S}}\circ \psi_{\hat{T}}
=\psi_{S}\circ \psi_{\hat{T}|\hat{S}}
=\psi_{S}\circ \psi_{T}
\]and the
proof is complete.
\end{proof}

The theorem shows that $V_I$ is a $\mathcal{U}_{I,I}$-right module. In
particular, the quantized walled Brauer algebra $\mathfrak{B}_{r,s}^n(q)$
acts on mixed tensor space $V^{\otimes r}\otimes
{V^*}^{\otimes s}$. 

Recall that $\mathfrak{B}_{m,0}^n(q)\cong
\mathcal{H}_m$ and $\mathfrak{B}_{r,s}^n(q)\cong
\mathcal{A}_{r,s}(R)$.
The Hecke algebra $\mathcal{H}_m$ acts on  ordinary tensor
space: 
Let the symmetric group $\mathfrak{S}_m$ act on $I(n,m)$ by place
permutation, i.~e. $\mathbf
i.s_k=(i_1,\ldots,i_{k-1},i_{k+1},i_k,i_{k+2},\ldots,i_m)$ for a
Coxeter generator $s_k$. Then the Hecke algebra $\mathcal{H}_m$ acts on
$V^{\otimes m}$ by
\[
v_{\mathbf i}T_k=\begin{cases}
  q^{-1}v_{\mathbf i}&\text{ if }i_k=i_{k+1}\\
  v_{\mathbf i.s_k}&\text{ if }i_k<i_{k+1}\\
  v_{\mathbf i.s_k}+(q^{-1}-q)v_{\mathbf i}&\text{ if }i_k>i_{k+1}
\end{cases}
\]
Similarly, there's an action of $\mathcal{A}_{r,s}(R)$ on mixed
tensor space given in \cite{leduc}. By checking the action of
generators of these algebras we obtain
\begin{prop}\label{prop:Hecke}
  \begin{enumerate}
  \item\label{item:Hecke}
    The action of $\mathfrak{B}_{m,0}^n(q)$ on ordinary tensor space 
    coincides
    with the action of the Hecke algebra $\mathcal{H}_m$.
  \item
    The action of $\mathfrak{B}_{r,s}^n(q)$ on mixed tensor space
    coincides with the
    action of  $\mathcal{A}_{r,s}(R)$ in
    \cite{leduc} ($q$ replaced by $q^{-1}$).
   \end{enumerate}
\end{prop}

\section{The Hopf algebra  $\mathbf U_R$}\label{section:hopf}
In this section, we
introduce the quantized enveloping algebra of the general linear Lie
algebra $\mathfrak{gl}_n$ over $R$ with parameter $q$ and summarize
some well known results, see for example \cite{hongkang,jantzen,lusztig}.

Let $P^\vee$ be the free $\mathbb{Z}$-module with basis $h_1, \ldots,
h_n$ and let $\varepsilon_1, \ldots, \varepsilon_n \in {P^\vee}^*$ be
the corresponding dual basis: $\varepsilon_i$ is given by
$\varepsilon_i(h_j):=\delta_{i,j}$ for $j=1, \ldots, n$, where
$\delta$ is the usual Kronecker symbol.  For $i=1,\ldots,n-1$ let
$\alpha_i\in {P^\vee}^*$ be defined by $\alpha_i := \varepsilon_i -
\varepsilon_{i+1}$. 
\begin{defn}
  The quantum general linear algebra $U_q(\mathfrak{gl}_n)$ is the
  associative $\mathbb{Q}(q)$-algebra with $1$ 
  generated by the elements
  $e_i, f_i$ $(i=1, \ldots, n-1)$ and $q^{h}$ $(h\in P^\vee)$
  with the defining relations
  \begin{eqnarray*}
    &&q^0=1,\quad q^hq^{h'}=q^{h+h'}\\
    &&q^he_iq^{-h}=q^{\alpha_i(h)}e_i,\quad 
    q^hf_iq^{-h}=q^{-\alpha_i(h)}f_i,\\
    &&e_if_j-f_je_i=\delta_{i,j}\frac{K_i-K_i^{-1}}{q-q^{-1}},\quad
    \mbox{ where }K_i := q^{h_{i}-h_{i+1}},\\
    &&e_i^2e_j-(q+q^{-1})e_ie_je_i+e_je_i^2=0\quad\mbox{ for }|i-j|=1,\\
    &&f_i^2f_j-(q+q^{-1})f_if_jf_i+f_jf_i^2=0\quad\mbox{ for }|i-j|=1,\\
    &&e_ie_j=e_je_i,\quad f_if_j=f_jf_i\quad\mbox{ for }|i-j|>1.
  \end{eqnarray*}
\end{defn}
$U_q(\mathfrak{gl}_n)$ is a Hopf algebra with comultiplication
$\Delta$, counit $\varepsilon$ the unique algebra homomorphisms and
antipode $S$ the unique invertible algebra anti-homomorphism defined on
generators by
\begin{eqnarray*}
  &&\Delta(q^h)=q^h\otimes q^h,\\
  &&\Delta(e_i)=e_i\otimes K_i^{-1}+1\otimes e_i,\quad
  \Delta(f_i)=f_i\otimes 1+K_i\otimes f_i,\\
  &&\varepsilon(q^h)=1,\quad \varepsilon(e_i)=\varepsilon(f_i)=0,\\
  && S(q^h)=q^{-h},\quad S(e_i)=-e_iK_i,\quad S(f_i)=-K_i^{-1}f_i .
\end{eqnarray*}
$U_q(\mathfrak{gl}_n)$ acts on our free $R$-module $V$ of rank $n$
(here we have $R=\mathbb{Q}(q)$) by
\begin{eqnarray*} 
  q^h v_j&=& 
  q^{\varepsilon_j(h)}v_j \mbox{ for }h\in P^\vee,\,j=1,\ldots,n,\\
  e_iv_j&=&\left\{
    \begin{array}{ll}
      v_i&\mbox{ if }j=i+1\\
      0&\mbox{ otherwise,}
    \end{array}
  \right. \hspace{3cm}
  f_iv_j=\left\{
    \begin{array}{ll}
      v_{i+1}&\mbox{ if }j=i\\
      0&\mbox{ otherwise.}
    \end{array}
  \right.
\end{eqnarray*}
We call $V$ the \emph{vector representation} of $U_q(\mathfrak{gl}_n)$.

Let $[l]_q!:=[l]_q[l-1]_q\ldots[1]_q$
and set 
$e_i^{(l)}:=\frac{e_i^l}{[l]_q!}$, $f_i^{(l)}:=\frac{f_i^l}{[l]_q!}$.
Let $\mathbf U$ be the $\mathbb{Z}[q,q^{-1}]$ subalgebra of  
$U_q(\mathfrak{gl}_n)$ generated by the $q^h$ and the divided powers 
$e_i^{(l)}$ and $f_i^{(l)}$ for $l\geq 0$. Then $\mathbf U$
is a Hopf algebra and we have
\begin{eqnarray*} 
  \Delta(e_i^{(l)})&=&
  \sum_{k=0}^l q^{k(l-k)}e_i^{(l-k)}\otimes K_i^{k-l} e_i^{(k)}\\
  \Delta(f_i^{(l)})&=&
  \sum_{k=0}^l q^{-k(l-k)}f_i^{(l-k)}K_i^k\otimes f_i^{(k)}\\
  S(e_i^{(l)})&=& (-1)^lq^{l(l-1)}e_i^{(l)}K_i^l\\
  S(f_i^{(l)})&=& (-1)^lq^{-l(l-1)}K_i^{-l}f_i^{(l)}\\
  \varepsilon(e_i^{(l)})&=&\varepsilon(f_i^{(l)})=0.
\end{eqnarray*}
Furthermore, the $\mathbb{Z}[q,q^{-1}]$-lattice
$V_{\mathbb{Z}[q,q^{-1}]}$ in $V$ generated by the $v_i$ is invariant
under the action of $\mathbf U$.  Recall that $R$ is a commutative
ring with $1$ and $q\in R$ invertible. Then $R$ is a
$\mathbb{Z}[q,q^{-1}]$-module via specializing $q\in
\mathbb{Z}[q,q^{-1}]\mapsto q\in R$.  Let $\mathbf
U_R:=R\otimes_{\mathbb{Z}[q,q^{-1}]}\mathbf U$. Then $\mathbf U_R$
inherits a Hopf algebra structure from $\mathbf U$ and
$V_R:=R\otimes_{\mathbb{Z}[q,q^{-1}]}V_{\mathbb{Z}[q,q^{-1}]}$ is a
$\mathbf U_R$-module.  If no ambiguity arises, we shall henceforth
write $V$ instead of $V_R$ and $\mathbf U$ instead of $\mathbf U_R$.

If $W_1$ and $W_2$ are $\mathbf U$-modules, one can turn $W_1\otimes
W_2$ into a $\mathbf U$-module by setting $x(w_1\otimes
w_2)=\Delta(x)(w_1\otimes w_2)$ for $x\in\mathbf U$ and $w_i\in W_i$.
Thus ordinary tensor space becomes a $\mathbf U$-module.  Let $\rho$
be the representation of $\mathbf U$ on ordinary tensor space.  One
easily checks that the action of the Hecke algebra and thus of
$\mathfrak{B}_{m,0}^n(q)$ on ordinary tensor space commutes with the
action of $\mathbf{U}$.

Green showed in \cite{green}, that $\rho(\mathbf{U})$ is the $q$-Schur
algebra defined in \cite{dipperdonkin,dipperjamesschur}.  Let
$\sigma:\mathcal{H}_{m}\to \mathrm{End}_R (V^{\otimes m})$ be the
representation of the Hecke algebra. The next theorem not only shows
that $\mathrm{End}_{{\mathbf U}} (V^{\otimes m})$ is an epimorphic
image of the Hecke algebra $\mathcal{H}_m$ (this was proved in full
generality in \cite[Theorem~6.2]{duparshallscott}), but also that some
kind of converse is true (\cite[Theorem~3.4]{dipperjames}). In the
literature, this property is called Schur--Weyl duality.

\begin{thm}[Schur--Weyl duality for ordinary 
  tensor space \cite{duparshallscott,dipperjames}]\label{thm:schurweyl}
  \begin{eqnarray*}
    \mathrm{End}_{{\mathbf U}}
    (V^{\otimes m})&=&\sigma(\mathcal{H}_{m}) \\
    \mathrm{End}_{\mathcal{H}_{m}}(V^{\otimes m})&=&
    \rho({\mathbf U}) . 
  \end{eqnarray*}
\end{thm}
One can turn $V^*$, and thus mixed tensor space, into a ${\mathbf
  U}$-module.  Our goal is to generalize part of the above theorem by
proving the first half of Schur--Weyl duality for mixed tensor space,
where the Hecke algebra is replaced by the `quantized walled Brauer
algebra' given in Definition \ref{defn:qwalledbrauer}.

\section{Dual $\mathbf U$-modules}
\label{section:finitedual}
Let $W$ be an $R$-free ${\mathbf U}$-module of finite rank.  The
dual module $W^*=\mathrm{Hom}_R(W,R)$ carries a ${\mathbf U}$-module
structure via $(xg)(v)=g(S(x)v)$ for $x\in {\mathbf U},g\in W^*,
v\in W$. Thus mixed tensor space $V^{\otimes r}\otimes {V^*}^{\otimes
  s}$ becomes  a ${\mathbf U}$-module.

Also, there is a different $\mathbf U$-action on the dual module
$\mathrm{Hom}_R(W,R)$ of $W$ given by $(xg)(v)=g(S^{-1}(x)v)$. This
$\mathbf U$-module will be denoted by $W'$.  Since $S^2$ is an inner
automorphism of $\mathbf U$, $W^*$ and $W'$
are isomorphic as $\mathbf U$-modules,
even in a natural way, but the
isomorphism is not the identity on the dual space.  Similarly, $W$ and
its double dual $W^{**}$ are naturally isomorphic as $\mathbf
U$-modules, but the isomorphism is not the usual isomorphism between
an $R$-module and its double dual (see \cite{jantzen}).

\begin{lem}[\cite{hongkang}, Lemma~3.5.1]\label{lem:doubledual}
The usual isomorphism between an $R$-free module $W$ of finite rank 
and its double dual induces  $\mathbf U$-module isomorphisms
\[
W'^*\cong W\cong {W^*}'.
\]
\end{lem}

The $R$-linear map $W'\otimes W\to \mathrm{End}_R(W)$ mapping 
$g\otimes v$ with $g\in W',v\in W$ to the $R$-endomorphism
of $W$ given by $(g\otimes v)(w)=g(w)\cdot v$
for $w\in W$ is an isomorphism.
Let the set of  \emph{${\mathbf U}$-invariants} 
of $W$ be given by
\[
W^{{\mathbf U}}:= 
\{w \in W: xw=\varepsilon(x)w
\mbox{ for all }x\in {\mathbf U} \}.
\]

The following lemma provides a connection
between ${\mathbf  U}$-invariants and ${\mathbf U}$-endomorphisms.

\begin{lem}\label{lem:V'otimesV}
  \[
  \mathrm{End}_{{\mathbf U}}(W)
  \cong(W'\otimes W)^{{\mathbf U}}
  \]
  under the above isomorphism.
\end{lem}
\begin{proof}
  In \cite{jantzen} it is shown that the natural isomorphism $
  \mathrm{End}_R(W)\cong W\otimes W^*$ induces an isomorphism
  between $\mathrm{End}_{{\mathbf  U}}(W)$ and $(W\otimes W^*)^{{\mathbf
      U}}$ if $R=\mathbb{Q}(q)$. One has to be more 
  careful when the ground ring $R$ is not
  $\mathbb{Q}(q)$, for example if $[n]_q=0$ for $n\geq 2$, then
  $e_i=f_i=0$ and ${\mathbf U}_R$ is not generated by the $e_i$, $f_i$
  and the $q^h$. In general, a result in 
 \cite{kemper} on Hopf algebras shows that there is an isomorphism  between
 $\mathrm{End}_{{\mathbf  U}}(W)$ and $(W\otimes W^*)^{{\mathbf
     U}}$. The claim of the lemma follows if one replaces $S$ by
 $S^{-1}$ and $\Delta$ by $\tau\circ\Delta$ in the proof in
 \cite{kemper} where $\tau$ is the algebra homomorphism $\tau:\mathbf
 U\otimes \mathbf U\to  \mathbf U\otimes \mathbf U: u_1\otimes u_2\mapsto
 u_2\otimes u_1$. 
\end{proof}
\begin{rem}
  The question may arise why we introduce a second kind of dual
  module $W'$. One could also show that  $W'$ and $W^*$ are isomorphic
  to obtain an isomorphism 
  $\mathrm{End}_{{\mathbf  U}}(W)\cong (W^*\otimes W)^{{\mathbf
      U}}$ or
  use the isomorphism
  $\mathrm{End}_{{\mathbf  U}}(W)\cong (W\otimes W^*)^{{\mathbf
      U}}$.
  The reason for  working  with invariants of $W'\otimes W$ is
  the following: 
  the isomorphism  $W'\cong W^*$ is not the  identity on linear
  forms, thus  $\mathrm{End}_{{\mathbf  U}}(W)\cong (W^*\otimes W)^{{\mathbf
      U}}$ is not the natural isomorphism.  If one chooses a basis for
  $W$, then one gets in a natural way  bases for dual modules, tensor
  products and endomorphim rings thereof. The isomorphism
  $\mathrm{End}_{{\mathbf  U}}(W)\cong (W'\otimes W)^{{\mathbf U}}$
  maps basis elements to basis elements while the isomorphism
  $\mathrm{End}_{{\mathbf  U}}(W)\cong (W^*\otimes W)^{{\mathbf 
      U}}$ does not. A similar argument shows that
  $\mathrm{End}_{{\mathbf  U}}(W)\cong (W\otimes W^*)^{{\mathbf  U}}$
  is not the right kind of invariants when $W$ is a tensor product of
  $V$'s and $V^*$'s, since one would have to use the isomorphism
  $W^{**}\cong W$ which is not the natural isomorphism between $W$ and
  its double dual. 
  Note that we will use the isomorphism of Lemma~\ref{lem:V'otimesV}
  explicitely and map elements 
  through this isomorphism.
\end{rem}


Next we have a result describing the interaction of duality with
tensor product.

\begin{lem}[\cite{kassel}]\label{lem:(VW)*}
  Let $W_1$ and $W_2$ be  ${\mathbf U}$-modules which are $R$-free 
  of finite rank. Then there are natural isomorphisms of $\mathbf U$-modules
  \[
  (W_1\otimes W_2)^*\cong W_2^*\otimes W_1^*,\qquad
  (W_1\otimes W_2)'\cong W_2'\otimes W_1'.
  \]
  The isomorphisms 
  $\phi^\Box:W_2^\Box\otimes W_1^\Box\rightarrow(W_1\otimes W_2)^\Box $
  (with $\Box=*$ or $'$)
  are given by
  \[
  \phi^\Box(g\otimes h)\;(v\otimes w):=h(v)g(w)
  \]
  for $g\in  W_2^\Box, h\in  W_1^\Box,v\in W_1,w\in W_2$.
\end{lem}

The following lemma exhibits isomorphisms $ V'\otimes V\to V\otimes
V^*$ and $ V^*\otimes V\to V\otimes V^*$. As mentioned in the remarks
preceding Lemma \ref{lem:doubledual} we have an isomorphism $V^*\cong
V'$ as $\mathbf{U}$-modules. We will need the following two
isomorphisms generalizing the interchange of the components of a
tensor product. Recall the basis $\{v_1^*,\ldots,v_n^*\}$ of $V^*$
dual to the basis $\{v_1,\ldots,v_n\}$ of $V$. In order to emphasize
that $v_i^*$ is considered as an element of $V'$, we write $v_i'$
instead of $v_i^*$.
\begin{lem}\label{lem:VV*V'V}
  Let $\psi':V'\otimes V\rightarrow V\otimes V^*$,
  and  $\psi:V^*\otimes V\rightarrow V\otimes V^*$ be the 
  homomorphisms of $R$-modules defined by
  \begin{eqnarray*}
    \psi'(v_i'\otimes v_k)&=&\left\{
      \begin{array}{ll}
        q^{n+1-2i}v_k\otimes v_i^*&k\neq i\\
        q^{n+1-2i}
        \left(q^{-1}v_i\otimes v_i^*+(q^{-1}-q)\sum_{l=1}^{i-1}v_l\otimes
          v_l^*\right)&k=i 
      \end{array}
    \right.\\
    \psi(v_i^*\otimes v_k)&=&\left\{
      \begin{array}{ll}
        v_k\otimes v_i^*&k\neq i\\
        q^{-1}v_i\otimes v_i^*+(q^{-1}-q)\sum_{l=1}^{i-1}v_l\otimes
          v_l^*&k=i 
      \end{array}
    \right.
  \end{eqnarray*}
  Then  $\psi'$ and $\psi$ 
  are  isomorphisms of 
  ${\mathbf U}$-modules.
\end{lem}
\begin{proof}
   Both maps are invertible since $q$ is invertible. Now it can be
  easily checked that they commute with the action of  ${\mathbf U}$.
\end{proof}
 Note that $\psi'$ can be obtained from $\psi$ by using the
  isomorphism $V^*\cong V'$. 

\section{An isomorphism and its tangle version}
\label{section:isomorphism} 

Now we apply the results on duality from the previous section to
obtain the following isomorphisms, which will be needed in the proof
of the main results.

\begin{prop}\label{prop:isomorphisms}
  \begin{enumerate}
  \item
  For $r>0$ there is an isomorphism of $R$-modules
  \[
  \mathrm{End}_R(V^{\otimes r}\otimes {V^*}^{\otimes s})\to
  \mathrm{End}_R( V^{\otimes r-1}\otimes {V^*}^{\otimes s+1})
  \]
  mapping $\mathbf U$-endomorphisms to $\mathbf U$-endomorphisms. 
  \item
    Repeated application of the above isomorphism leads to an
    isomorphism
  \[
  \mathrm{End}_{{\mathbf U}}
  (V^{\otimes r+s})\to
  \mathrm{End}_{{\mathbf U}}
  (V^{\otimes r}\otimes {V^*}^{\otimes s})\]
  of $R$-modules. 
\end{enumerate}
\end{prop}

\begin{proof}
$\mathrm{End}_R(V^{\otimes r}\otimes {V^*}^{\otimes s})$ is isomorphic
as an $R$-module
to  $(V^{\otimes r}\otimes {V^*}^{\otimes s})'\otimes
(V^{\otimes r}\otimes {V^*}^{\otimes s})$, such that $\mathbf
U$-endomorphisms correspond to $\mathbf U$-invariants by 
Lemma~\ref{lem:V'otimesV}. Then Lemma~\ref{lem:(VW)*} and
Lemma~\ref{lem:doubledual}
imply that 
$(V^{\otimes r}\otimes {V^*}^{\otimes s})'\otimes
(V^{\otimes r}\otimes {V^*}^{\otimes s})$ is isomorphic as a $\mathbf
U$-module to 
$ V^{\otimes s}\otimes {V'}^{\otimes r}\otimes
V^{\otimes r}\otimes  {V^*}^{\otimes s}=
V^{\otimes s}\otimes {V'}^{\otimes r-1}\otimes V'\otimes V\otimes
V^{\otimes r-1}\otimes  {V^*}^{\otimes s}
$. Applying  $\psi'$ to the middle part, this is isomorphic as a
$\mathbf U$-module to 
$ V^{\otimes s}\otimes {V'}^{\otimes r-1}
\otimes V\otimes V^*\otimes
V^{\otimes r-1}\otimes  {V^*}^{\otimes s}$. By a similar argument as
above, we have $R$-isomorphisms
\[ 
V^{\otimes s}\otimes {V'}^{\otimes r-1}
\otimes V\otimes V^*\otimes
V^{\otimes r-1}\otimes  {V^*}^{\otimes s}
\cong 
\mathrm{End}_R(V^*\otimes 
V^{\otimes r-1}\otimes {V^*}^{\otimes s})
\]
and again $\mathbf U$-invariants map to $\mathbf U$-endomorphisms. 
Successive application of $\psi$ shows that 
$V^*\otimes 
V^{\otimes r-1}\otimes {V^*}^{\otimes s}$ and $V^{\otimes r-1}\otimes
{V^*}^{\otimes s+1}$ are isomorphic $\mathbf U$-modules (and thus
isomorphic as $R$-modules), so we have
\begin{eqnarray*}
\mathrm{End}_R(V^*\otimes 
V^{\otimes r-1}\otimes {V^*}^{\otimes s})&\cong&
\mathrm{End}_R( V^{\otimes r-1}\otimes {V^*}^{\otimes s+1})\\
\mathrm{End}_{\mathbf U}(V^*\otimes 
V^{\otimes r-1}\otimes {V^*}^{\otimes s})&\cong&
\mathrm{End}_{\mathbf U}( V^{\otimes r-1}\otimes {V^*}^{\otimes s+1}).
\end{eqnarray*}
The second part follows by  iterated application of the first part. 
\end{proof}

So far, we didn't use the particular form of the isomorphisms in
Lemma~\ref{lem:V'otimesV} and
Proposition~\ref{prop:isomorphisms}. 
Let $T$ be an oriented tangle   of type $(\downarrow^r\uparrow^s)$. 
Then the isomorphism in  Proposition~\ref{prop:isomorphisms} maps
$\psi_T$  (see Definition \ref{psiT})  to an element of
$\mathrm{End}_R( V^{\otimes r-1}\otimes {V^*}^{\otimes s+1})$.
Our next aim is to show that there is a tangle
$S$ of type $(\downarrow^{r-1}\uparrow^{s+1})$ such that $\psi_T$ maps
to $\psi_S$.

\begin{lem}
  \label{lemma:cross}
    Let  $\psi$ be  as in Lemma~\ref{lem:VV*V'V}.
    Let $\widehat{S}=\raisebox{-.1cm}{\epsfbox{si.18}}$,
    then 
    \[
    \psi_{\widehat{S}}=\mathrm{id}^{\otimes k-1}\otimes 
    \psi\otimes \mathrm{id}^{\otimes m-k-1}
    \]
  
\end{lem}
\begin{proof}
 One just has to check that the action of $\psi_{\widehat{S}}$ is exactly
 the same as the action of $\psi$.
\end{proof}

Suppose that  $T$ is an oriented tangle of type $(I,J)$ such that
$I_1=J_1=\downarrow$. As shown in Figure~\ref{figure:psi}, the oriented tangle
$T$ represented by the rectangle can be
embedded into a piece of knot diagram (the vertical strands having
suitable orientation). The resulting oriented tangle is an oriented tangle of type
$((\uparrow,I_2,\ldots,I_m),(\uparrow,J_2,\ldots,J_m))$ and will be
denoted by $\raisebox{-.2cm}{\epsfbox{crossings.11}}T$.  
\begin{figure}[h]
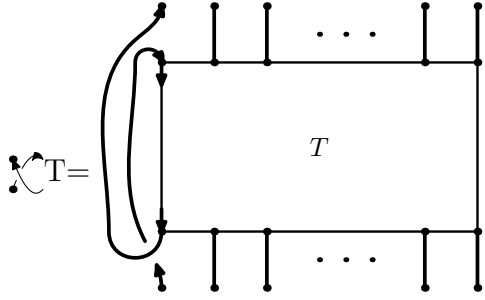

  \centerline{\raisebox{-.2cm}{\epsfbox{crossings.11}}T=
  \raisebox{-1.5cm}{\epsfbox{psi.1}}}
\caption{$T$ embedded into a piece of knot diagram}
\label{figure:psi}
\end{figure}

\begin{lem}\label{lem:flip}
Let $T$ be an oriented tangle of type $(I,J)$ with $I_1=J_1=\downarrow$.
 If we identify $\psi_T$ with
an element of an $2m$-fold tensor space via the map
$\mathrm{Hom}_R(V_I,V_J)\cong V_{I}'\otimes V_J$, we get with
$S=\raisebox{-.2cm}{\epsfbox{crossings.11}}T$: 
\[
\mathrm{id}^{\otimes
  m-1}\otimes\psi'\otimes\mathrm{id}^{\otimes m-1}(\psi_T)=
\psi_{S}.
\]
The same holds if $T$ (and in the same way $S$) is replaced by a linear
combination of such tangles. 
\end{lem}
\begin{proof}
  For  $\mathbf i\in I(n,m)$, let $\bar{\mathbf i}=
  (i_2,\ldots,i_m)\in I(n,m-1)$, 
  $\bar{\bar{\mathbf i}}=(i_3,\ldots,i_m)\in
  I(n,m-2)$ and let $\widehat{\mathbf i}=(i_m,\ldots,i_2)$ be the multi
  index $\bar{\mathbf i}$ in reversed order. Let $\bar
  I=(I_2,\ldots,I_m)$, $\widehat{I}$, etc.~ be defined similarly.
  Then $\psi_T$
  corresponds to 
  $\sum_{\mathbf i,\mathbf j}\psi_{\mathbf i,\mathbf j}(T)v_{\widehat{\mathbf
      i}}^{\widehat{I}}\otimes v_{i_1}'\otimes v_{j_1}\otimes
  v_{\bar{\mathbf j}}^{\bar J}$ under
  the isomorphism $\mathrm{Hom}_R(V_I,V_J)\cong V_{I}'\otimes
  V_J$. Under $\mathrm{id}^{\otimes
    m-1}\otimes\psi'\otimes\mathrm{id}^{\otimes m-1}$, 
  this $2m$-fold tensor maps to 
  \begin{align*}
    &\sum_{\mathbf i,\mathbf j,i_1\neq j_1}\psi_{\mathbf i,\mathbf j}(T)
    q^{n+1-2i_1}v_{\widehat{\mathbf
        i}}^{\widehat{I}}\otimes v_{j_1}\otimes v_{i_1}^*\otimes v_{\bar{\mathbf
      j}}^{\bar J}\\
   & + \sum_{\mathbf i,\mathbf j,i_1= j_1}
    \psi_{\mathbf i,\mathbf j}(T)
    q^{n+1-2i_1}v_{\widehat{\mathbf
        i}}^{\hat I}\otimes \left(q^{-1}v_{i_1}\otimes v_{i_1}^*
    +(q^{-1}-q)\sum_{l=1}^{i_1-1}v_l\otimes v_l^*\right)\otimes 
  v_{\bar{\mathbf j}}^{\bar{J}}.
  \end{align*}
  Comparing coefficients, we see that this corresponds to $\psi_S$ if
  and only if 
  \begin{equation}\label{equation:psi'}
  \psi_{\mathbf i,\mathbf j}(S)
  =\left\{
    \begin{array}{ll}
      q^{n+1-2j_1}\psi_{(j_1,\bar{\mathbf i}),(i_1,\bar{\mathbf j})}(T)&
      \mbox{ if }i_1\neq j_1\bigskip\\
      q^{n-2i_1} \psi_{\mathbf i,\mathbf j}(T)\\
      +(q^{-1}-q)\sum\limits_{k=i_1+1}^n
      q^{n+1-2k} \psi_{(k,\bar{\mathbf i}),(k,\bar{\mathbf j})}(T)
      & \mbox{ if }i_1= j_1.
    \end{array}
  \right. 
  \end{equation}
This is an easy computation for $m=1$, since there is only
one descending oriented tangle. So assume
that $m\geq 2$. Note that we can assume that $T$ is descending with
respect to $(\mathbf i,\mathbf j)$.
We will show Equation~\eqref{equation:psi'}
using Theorem~\ref{thm:linking}.
Let
\begin{eqnarray*}
  (I^\circlearrowleft,J^\circlearrowleft)&=&
  ((I_2,\ldots,I_m,\bar{J_m}),(\bar{I_1},J_1,\ldots,J_{m-1})),\\ 
  (I^\circlearrowright,J^\circlearrowright)&=&
  ((\bar{J_1},I_1,\ldots,I_{m-1}),(J_2,\ldots,J_m,\bar{I_m}))
\end{eqnarray*}
where $\bar{I_k}=\downarrow$ if and only if $I_k=\uparrow$. 
Let
$(\mathbf i^\circlearrowleft,\mathbf j^\circlearrowleft)$ and 
$(\mathbf i^\circlearrowright,\mathbf j^\circlearrowright)$
be defined similarly (without the $\bar{}$ ).
We can define an isomorphism of
$R$-modules
$\circlearrowleft:\mathcal{U}_{I,J}\rightarrow 
\mathcal{U}_{I^\circlearrowleft,J^\circlearrowleft}$
by rotating the vertices 
anticlockwise, similarly $\circlearrowright:\mathcal{U}_{I,J}\rightarrow 
\mathcal{U}_{I^\circlearrowright,J^\circlearrowright}$ 
rotates the vertices clockwise (see Figure~\ref{figure:rotation_6}).
Let 
$\widehat{S}_1^{\swarrow\!\!\!\!\!\!\nwarrow}=
\raisebox{-.1cm}{\epsfbox{si.10}}
=\raisebox{-.1cm}{\epsfbox{si.11}}
+(q^{-1}-q)
\raisebox{-.1cm}{\epsfbox{si.12}}=S_1^{\swarrow\!\!\!\!\!\!\nwarrow}
+(q^{-1}-q)E_1^{\leftleftarrows}$. 
Let $T'=\circlearrowleft(T)$, and
$T''=T'|\widehat{S}_1^{\swarrow\!\!\!\!\!\!\nwarrow}$, then we 
have 
$\circlearrowright(T'')
=\raisebox{-.2cm}{\epsfbox{crossings.11}}T=:S$.
\begin{figure}[h]
\centerline{\epsfbox{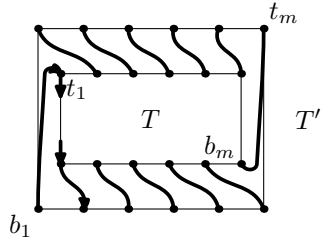}}
\caption{a rotated oriented tangle}\label{figure:rotation_6}
\end{figure}
The coefficients $\psi_{\mathbf i^\circlearrowleft,\mathbf
  j^\circlearrowleft}(T')$
are closely related to
the coefficients $\psi_{\mathbf i,\mathbf j}(T)$: 
$T'$ is
descending with
respect to $(\mathbf i^\circlearrowleft,\mathbf j^\circlearrowleft)$
if and only if 
$T$ is descending with
respect to 
$(\mathbf i,\mathbf j)$.
The crossings in both oriented tangles yield the same factors,  so we only 
have to care about the  horizontal
strands from left to right.
Taking these strands into account, we get

\begin{eqnarray*}
\psi_{\mathbf i^\circlearrowleft,\mathbf j^\circlearrowleft}(T')&=&q^{-2i_1+n+1} 
\psi_{\mathbf i,\mathbf j}(T)\mbox{ if }J_m=\uparrow.\\
\psi_{\mathbf i^\circlearrowleft,\mathbf j^\circlearrowleft}(T')&=&
q^{2(j_m-i_1)}
\psi_{\mathbf i,\mathbf j}(T)\mbox{ if }J_m=\downarrow.
\end{eqnarray*}
Similarly, one can show that
\begin{eqnarray*}
  \psi_{\mathbf i,\mathbf j}(S)&=&
  \psi_{\mathbf i^\circlearrowleft,\mathbf j^\circlearrowleft}(T'')
  \mbox{ if }J_m=\uparrow\\
  \psi_{\mathbf i,\mathbf j}(S)&=&q^{n+1-2j_m}
  \psi_{\mathbf i^\circlearrowleft,\mathbf j^\circlearrowleft}(T'')
  \mbox{ if }J_m=\downarrow.
\end{eqnarray*}
Suppose first, that $i_1=j_1$. Then $\psi_{\mathbf k,\mathbf
  j^\circlearrowleft}(\widehat{S}_1^{\swarrow\!\!\!\!\!\!\nwarrow})=q^{-1}$ if $\mathbf k
=\mathbf j^\circlearrowleft$, 
$\psi_{\mathbf k,\mathbf
  j^\circlearrowleft}(\widehat{S}_1^{\swarrow\!\!\!\!\!\!\nwarrow})=q^{-1}-q$
if $\mathbf k=(k,k,\bar{\bar{\mathbf j^\circlearrowleft}})$ and 
$\psi_{\mathbf k,\mathbf
  j^\circlearrowleft}(\widehat{S}_1^{\swarrow\!\!\!\!\!\!\nwarrow})=0$
otherwise. If $J_m=\uparrow$, 
then we have
\begin{eqnarray*}
  \psi_{\mathbf i,\mathbf j}(S)&=&
  \psi_{\mathbf i^\circlearrowleft,\mathbf j^\circlearrowleft}(T'')
  =\sum_{\mathbf k}\psi_{\mathbf i^\circlearrowleft,\mathbf k}(T')
  \psi_{\mathbf k,\mathbf
    j^\circlearrowleft}(\widehat{S}_1^{\swarrow\!\!\!\!\!\!\nwarrow})\\ 
  &=& q^{-1}\psi_{\mathbf i^\circlearrowleft,\mathbf
    j^\circlearrowleft}(T')
  +(q^{-1}-q)\sum_{k>i_1}\psi_{\mathbf i^\circlearrowleft,
    (k,k,\bar{\bar{\mathbf j^\circlearrowleft}})}(T')\\
  &=& q^{n-2i_1}\psi_{\mathbf i,\mathbf j}(T)
  +(q^{-1}-q)\sum_{k>i_1}q^{-2k+n+1}\psi_{(k,\bar{\mathbf i}),
    (k,\bar{\mathbf j})}(T).
\end{eqnarray*}.

All other cases can be shown similarly.
\end{proof}

\begin{cor}\label{cor:tangleisom}
  Let $T\in\mathfrak{B}_{r,s}^n(q)$ be an oriented tangle. Then
  the isomorphism $\mathrm{End}_R(V^{\otimes r}
  \otimes {V^*}^{\otimes  s})\to
  \mathrm{End}_R( V^{\otimes r-1}\otimes {V^*}^{\otimes s+1})$ maps
  $\psi_T$ to $\psi_S$ where $S$ is the oriented tangle 
  \[S=\raisebox{-2cm}
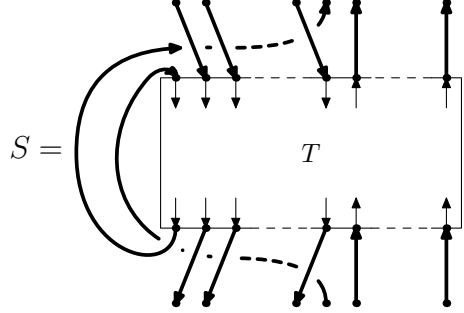
{\epsfbox{isom.1}}\]
\end{cor}
\begin{proof}
  We can split the isomorphism into two isomorphisms 
  \begin{align*}
    &\mathrm{End}_R(V^{\otimes r}\otimes {V^*}^{\otimes s})
    \to \mathrm{End}_R(V^*\otimes 
    V^{\otimes r-1}\otimes {V^*}^{\otimes s})\text{
      and }\\ 
    &\mathrm{End}_R(V^*\otimes 
    V^{\otimes r-1}\otimes {V^*}^{\otimes s})
    \to
    \mathrm{End}_R
    (V^{\otimes r-1}\otimes {V^*}^{\otimes s+1}).
  \end{align*} 
  Lemma~\ref{lem:flip} shows that $\psi_T$ is mapped to 
  $\psi_{\raisebox{-.5mm}{\epsfxsize= .2cm \epsfbox{crossings.11}}T}\in
  \mathrm{End}_R(V^*\otimes V^{\otimes r-1}\otimes {V^*}^{\otimes
    s})$ under the first isomorphism.
  
  Let $\psi_k=\mathrm{id}^{\otimes k-1}\otimes 
  \psi\otimes \mathrm{id}^{\otimes m-k-1}$
  and let $T_k$ be the tangle from Lemma~\ref{lemma:cross}
  such that $\psi_k=\psi_{T_k}$.
  Hence the second isomorphism maps 
  $\psi_{\raisebox{-.5mm}{\epsfxsize= .2cm \epsfbox{crossings.11}}T}$
  to $\psi_S$ where
  \[S=T_{r-1}^{-1}|\ldots |T_2^{-1}|T_1^{-1}|
  \raisebox{-2mm}{\epsfbox{crossings.11}}T |T_1|T_2\ldots| T_{r-1}.\]
\end{proof}
Thus $S$ is obtained by embedding $T$ into a piece of knot diagram. 
Note that $T$ can be obtained back
from $S$ in a similar way. 

\section{Schur--Weyl duality}
We are finally in a position to prove the main result, which shows
that the $\mathbf U$-invariants of the endomorphism algebra of mixed
tensor space are generated by the action of $\mathfrak{B}_{r,s}^n(q)$.

\begin{thm}[first part of Schur--Weyl
  duality for mixed tensor space]\label{thm:schurweyl1}
  Let
  $\sigma_{r,s}:\mathfrak{B}_{r,s}^n(q)\to 
  \mathrm{End}_R(V^{\otimes r}\otimes 
  {V^*}^{\otimes s})$ be the representation of the 
  quantized walled
  Brauer algebra, then
  \[
  \mathrm{End}_{{\mathbf U}}(V^{\otimes r}\otimes 
  {V^*}^{\otimes s})=\sigma_{r,s}(
  \mathfrak{B}_{r,s}^n(q)).
  \]
\end{thm}
\begin{proof}
  We fix a nonnegative integer $m$ and show the result for $r$ and $s$
  with $r+s=m$ 
  by induction on $s$.  For $s=0$ the claim
  follows from
  Proposition~\ref{prop:Hecke} and 
  Theorem~\ref{thm:schurweyl}.
  Assume that the theorem holds for $s<m$ and thus $r\geq 1$.  
  Recall the  $R$-linear isomorphism (see
  Proposition~\ref{prop:isomorphisms}): 
  \[    \mathrm{End}_{{\mathbf U}}(V^{\otimes r}\otimes {V^*}^{\otimes s})
  \cong 
  \mathrm{End}_{{\mathbf U}}( V^{\otimes r-1}\otimes {V^*}^{\otimes s+1})
  \]
  Since the theorem holds for $s$ we have
  $\psi_T\in\mathrm{End}_{{\mathbf U}} 
  (V^{\otimes r}\otimes {V^*}^{\otimes s})$ for any oriented tangle of type
  $((\downarrow^r\uparrow^s),(\downarrow^r\uparrow^s))$ and any element
  of $\mathrm{End}_{{\mathbf U}} 
  (V^{\otimes r}\otimes {V^*}^{\otimes s})$ is a linear combination of
  such $\psi_T$'s.
  By Corollary~\ref{cor:tangleisom} the image of $\psi_T$ under the
  isomorphism is in $\sigma_{r-1,s+1}(
  \mathfrak{B}_{r-1,s+1}^n(q))$. Furthermore each $\psi_S\in \sigma_{r-1,s+1}(
  \mathfrak{B}_{r-1,s+1}^n(q))$ can be obtained as the image of some
  $\psi_T$. This shows that the result holds for $(r-1,s+1)$.
\end{proof}

We also have the following.

\begin{cor}\label{cor:schurweyl1}
  \begin{enumerate}
  \item
    Let $I$ and $J$ be $m-tuples$ with entries in 
    $\{\downarrow, \uparrow\}$, such that the numbers of entries equal
    to $\uparrow$ coincide for $I$ and $J$. Then
    \[\mathrm{Hom}_{{\mathbf U}}(V_I,V_J)=\sigma(\mathcal{U}_{I,J})\] with
    $\sigma:\mathcal{U}_{I,J}\to
    \mathrm{Hom}_R(V_I,V_J):T\mapsto \psi_T$. 
  \item We have an isomorphism
    \[\mathrm{ann}_{\mathcal{H}_{r+s}}(V^{\otimes{r+s}})
    \cong
    \mathrm{ann}_{\mathfrak{B}_{r,s}^n(q)}
    (V^{\otimes{r}}\otimes
    {V^*}^{\otimes s})
    \]
    as $R$-modules.  In particular, the action of
    $\mathfrak{B}_{r,s}^n(q)$ is faithful if and only if
    $\mathrm{dim}(V)\geq r+s$. Furthermore, for any $r,s$
    the annihilator $ \mathrm{ann}_{\mathfrak{B}_{r,s}^n(q)}
    (V^{\otimes{r}}\otimes {V^*}^{\otimes s})$ is free over $R$ of
    rank not depending on $R$.
  \end{enumerate}
\end{cor}
\begin{proof}
1. $V_I$ and $V_J$ are isomorphic to $V^{\otimes r}\otimes
  {V^*}^{\otimes s}$, the isomorphisms are products of isomorphism
  $\psi_k$ and $\psi_k^{-1}$ as in the proof of Corollary~\ref{cor:tangleisom}.
  Multiplying the corresponding oriented tangles
  from the left resp.~right maps  
  $\mathfrak{B}_{r,s}^n(q)$ to $\mathcal{U}_{I,J}$. 

2.  The proof of Theorem~\ref{thm:schurweyl1}
  shows that there is an $R$-linear isomorphism from 
  $\mathcal{H}_{r+s}\to\mathfrak{B}_{r,s}^n(q)$ such that 
  $T\in\mathcal{H}_{r+s}\mapsto S\in \mathfrak{B}_{r,s}^n(q)$
  implies $\psi_T\mapsto \psi_S$ under the isomorphism 
  $\mathrm{End}_{{\mathbf U}}
  (V^{\otimes r+s})\to
  \mathrm{End}_{{\mathbf U}}
  (V^{\otimes r}\otimes {V^*}^{\otimes s})$. 
  But then $\psi_T=0$ if and only if
  $\psi_S=0$. The rest of the corollary follows from \cite{haerterich}.
\end{proof}

\begin{rem}
  The isomorphism $\mathrm{End}_{{\mathbf U}}(V^{\otimes r+s})\cong
  \mathrm{End}_{{\mathbf U}}(V^{\otimes r}\otimes {V^*}^{\otimes s})$
  can be modified such that the isomorphism maps $\psi_T$ to $\psi_S$
  with $S$ given as in Figure~\ref{figure:modified2}.
  \begin{figure}[h]
    \centerline{
      S=  \raisebox{-2cm}{\epsfbox{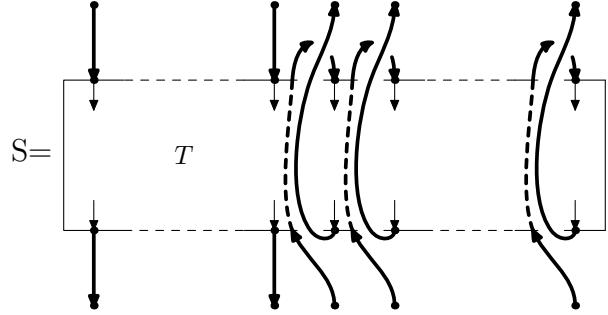}}}
    \caption{$\mathrm{End}_{{\mathbf U}}(V^{\otimes r+s})\cong
  \mathrm{End}_{{\mathbf U}}(V^{\otimes r}\otimes
  {V^*}^{\otimes s})$}\label{figure:modified2}
  \end{figure}
  In the classical case $q=1$, this isomorphism turns a permutation
  diagram (one having only vertical edges) into a walled Brauer
  diagram by rotating the right side of the Brauer diagram by
  $180^\circ$ around a horizontal axis.
\end{rem}


\end{document}